\documentclass[12pt,a4paper]{article}

\usepackage{graphicx}
\usepackage[fleqn]{amsmath}
\usepackage{amsfonts,bm,algorithm,algorithmic,amsthm}
\DeclareMathOperator*{\argmax}{argmax}
\newcommand{\E}{\mathrm{E}}
\newcommand{\so}{\mathrm{o}}
\newcommand{\lo}{\mathrm{O}}
\newcommand{\ep}{\epsilon}
\newcommand{\de}{\delta}
\newcommand{\id}{\mathbb{I}}
\newcommand{\dinf}{D_{\mathrm{inf}}}
\newcommand{\dmin}{D_{\mathrm{min}}}
\newcommand{\hen}[2]{\frac{\partial #1}{\partial #2}}
\newcommand{\n}{\nonumber}
\newcommand{\nn}{\nonumber\\}
\newcommand{\scap}{\,\cap\,}
\newcommand{\jn}{J_n}
\newcommand{\fhat}{\hat{F}}
\newcommand{\muhat}{\hat{\mu}}
\newcommand{\fhatn}[2]{\fhat_{#1}(#2)}
\newcommand{\fhatt}[2]{\fhat_{#1,#2}}
\newcommand{\muhatn}[2]{\muhat_{#1}(#2)}
\newcommand{\muhatt}[2]{\muhat_{#1,#2}}
\newcommand{\fb}{\mathcal{A}}
\newcommand{\mf}{\mathcal{F}}
\newcommand{\mx}{\mathcal{X}}
\newcommand{\rd}{\mathrm{d}}
\newcommand{\mom}{\frac{1}{-\mu}}
\newcommand{\moms}{\mbox{$\mom$}}
\newcommand{\momss}{-1/\mu}
\newcommand{\unu}{\underline{\nu}}
\newcommand{\numin}{\frac{\mu-\E(F)}{-\mu(1+\mu)}}
\newcommand{\numins}{(\mu-\E(F))/(-\mu(1+\mu))}
\newcommand{\onu}{\overline{\nu}}
\newcommand{\supp}{\mathrm{supp}}
\newcommand{\suppp}[1]{\{0\}\cup\,\supp(#1)}
\newcommand{\suppF}[1]{\supp'(#1)}
\newcommand{\ff}[1]{\fb_{\suppF{#1}}}
\newcommand{\fs}[1]{\fb_{#1}}
\newcommand{\di}[1]{\dmin(F_{#1},\mu^*)}

\newtheorem{theorem}{Theorem}
\newtheorem{lemma}[theorem]{Lemma}
\newtheorem{corollary}{Corollary}

\title{An Asymptotically Optimal Policy for Finite Support Models 
in the Multiarmed Bandit Problem}
\author{Junya HONDA and Akimichi TAKEMURA \medskip \\
        Department of Mathematical Informatics\\
        Graduate School of Information Science and Technology\\
        The University of Tokyo\\
        \texttt{\normalsize \{honda,takemura\}@stat.t.u-tokyo.ac.jp}}
\date{
February, 2010
}

\begin{document}
\maketitle

\begin{abstract}
Multiarmed bandit problem is an example of a dilemma between exploration
 and exploitation in reinforcement learning.
This problem is expressed as a model of a gambler playing a slot machine
 with multiple arms.
A policy chooses an arm so as to minimize the number of times that arms
 with inferior expectations are pulled.
We propose minimum empirical divergence (MED) policy and
prove asymptotic optimality of the policy  
for the case of finite support models.
In a setting similar to ours, Burnetas and Katehakis
have already proposed an asymptotically optimal policy.
However we do not assume knowledge of the specific support except for
 the upper and lower bounds of the support.
Furthermore, the criterion for choosing an arm, minimum empirical
 divergence,
can be computed easily by a convex optimization technique.
We confirm by simulations that MED policy demonstrates good
performance in finite time in comparison to other currently popular
 policies.

\end{abstract}

\section{Introduction}
\label{section-intro}
The multiarmed bandit problem is a problem based on an analogy with
playing a slot machine with more than one arm or lever.  Each arm has a reward
distribution and the objective of a gambler is to maximize the
collected sum of rewards by choosing an arm to pull for each round.
There is a dilemma between exploration and exploitation, namely the
gambler can not tell whether an arm is optimal unless he pulls it
many times, but it is also a loss to pull an inferior (i.e.\ non-optimal) arm 
many times.

We consider an infinite-horizon $K$-armed bandit problem.
There are $K$ arms $\Pi_1,\\\dots,\Pi_K$ and arms are pulled infinite number of 
times.
$\Pi_j$ has a probability distribution $F_j$ with the expected value
$\mu_j$ and the player receives a reward according to $F_j$ independently in
each round. If the expected values are known, it is optimal to always
pull the arm
with the maximum expected value $\mu^*=\max_j \mu_j$.
A policy is an algorithm to choose the next arm  to pull
based on the results of past rounds.

This problem is first considered by Robbins \cite{robbins}.
Since then, many studies have been conducted for the problem
\cite{agrawal,even,meuleau,strens,vermorel,yakowitz}.
There are also many extensions for the problem.
For example, Auer et al. \cite{adversarial} removed the assumption
 that rewards are
stochastic, and for the stochastic setting, the 
case of non-stationary distributions \cite{gittins,ishikida,katehakis},
or the case of infinite (possibly uncountable) arms
\cite{conti1,conti2} have been considered.

In our setting, 
Lai and Robbins \cite{lai} established a theoretical framework for 
determining optimal policies, and 
Burnetas and Katehakis \cite{burnetas} extended  their result to
multiparameter or non-parametric models.
Consider a model $\mf$, a generic family of distributions.
The player knows $\mf$ and that $F_j$ is an element of $\mf$.
Let $T_j(n)$ denote the number of times that $\Pi_j$ has been pulled 
over the first $n$ rounds.
A policy is {\em consistent} on model $\mathcal{F}$ if 
$\E[T_i(n)]=\so(n^a)$ for all
inferior arms $\Pi_i$ and all $a>0$.

Burnetas and Katehakis \cite{burnetas} proved the following
 lower bound for any inferior
arm $\Pi_i$ under consistent policy:
\begin{eqnarray}
\label{eq:optimal-bound}
T_i(n) \ge \left(\frac{1}{\inf_{G\in\mathcal{F}:\E(G)>\mu^*}D(F_i||G)}+\so(1)\right)\log n
\end{eqnarray}
with probability tending to one,
where $\E(G)$ is the expected value of distribution $G$ and
 $D(\cdot||\cdot)$ denotes the Kullback-Leibler divergence.
Under mild regularity conditions on $\mf$,
\begin{eqnarray}
\inf_{G\in\mathcal{F}:\E(G)>\mu}D(F||G)=
\min_{G\in\mathcal{F}:\E(G)\ge\mu}D(F||G)\n
\end{eqnarray}
and we write 
\begin{eqnarray}
\dmin(F,\mu)=
\min_{G\in\mathcal{F}:\E(G)\ge\mu}D(F||G)\n
\end{eqnarray}
 in the following.

A policy is asymptotically optimal if the expected value of
$T_j(n)$ achieves the right-hand side of 
(\ref{eq:optimal-bound}) as $n\rightarrow\infty$.
In \cite{lai} and \cite{burnetas}, policies achieving the above
bound are also
proposed.
These policies are based on the notion of {\em upper confidence bound}.
It can be interpreted as the upper confidence limit 
for the expectation of each arm with the significance 
level $1/n$.

Although policies based on upper confidence bound are optimal,
upper confidence bounds are often hard to compute in practice.
Then, Auer et al. \cite{ucb} proposed some policies called UCB.
UCB policies estimate the expectation of each arm in a similar way to upper
confidence bound.
They are practical policies for their simple form and fine performance.
Especially, ``UCB-tuned'' is widely used because of its excellent simulation
results.
However, UCB-tuned has not been analyzed theoretically and it is unknown
whether the policy has consistency.
Theoretical analyses of other UCB policies have been given,
but their coefficients of the logarithmic term do
not necessarily achieve the bound \eqref{eq:optimal-bound}. 

In this paper
we propose minimum empirical divergence (MED) policy. 
We prove the asymptotic optimality of MED
when the model $\mf$ is the family of distributions with a finite bounded
support, denoted by $\fb$.
This model consists of all distributions with finite supports over a
given interval, e.g. $[-1,0]$.
It is larger than the model used in \cite{burnetas}, which assumes a
specific finite support.
We also demonstrate simulation results of MED policy comparable to
UCB policies.

Our MED policy is motivated by the observation of
\eqref{eq:optimal-bound}.
When a policy achieving \eqref{eq:optimal-bound} is used, an inferior arm
$\Pi_i$ waits roughly
 $\exp(n_i\dmin(F_i,\mu^*))$ rounds to be pulled after
the $n_i$-th play of $\Pi_i$.
Then, it can be expected that a policy pulling $\Pi_i$ with probability
$\exp(-n_i\dmin(F_i,\mu^*))$ will achieve \eqref{eq:optimal-bound}.
MED policy is obtained by
plugging $\hat{F}_i,\hat{\mu}^*$ into $F_i,\mu^*$ in $\dmin$,
where $\hat{F}_i$ is the
empirical distribution of rewards from $\Pi_i$ and $\hat{\mu}^*$
 is the current best
sample mean.

MED policy requires a computation of
$\dmin(\hat{F}_i,\hat{\mu}^*)=\min_{G\in\fb:\E(G)\ge\hat{\mu}^*}
D(\hat{F}_i||G)\n$ at each round 
whereas upper
confidence bound requires the computation of
\begin{eqnarray}
\max_{G\in\fb:\dmin(\hat{F}_i||G)\le
\frac{\log n}{n_i}}\E(G).\label{emax}
\end{eqnarray}
$\dmin$ and \eqref{emax} are quantity dual to each other but the former
has two advantages in practical implementation.
First, $\dmin(\hat{F}_i,\hat{\mu}^*)$ is smooth in $\hat{\mu}^*$
 which converges to $\mu^*$.
Therefore the value in the previous round can be used as a good
approximation of $\dmin$ for the current round.
On the other hand \eqref{emax} continues to increase according to $n$
and it has to be computed many times.
Second, as shown in Theorem \ref{thm8} below,
 $\dmin$ can be expressed as a {\it univariate} convex optimization
problem for our model $\fb$.
Although \eqref{emax} is also a convex optimization problem, the
nonlinear constraint $D(\hat{F}_i||G)\le \frac{\log n}{n_i}$ is
harder to handle.

MED policy is categorized as a probability matching method (see,
e.g. \cite{vermorel} for classification of policies).
In this method each arm is pulled according to the probability
reflecting how likely the arm is to be optimal.
For example, Wyatt \cite{wyatt} proposed probability matching policies for
Boolean and Gaussian models by Bayesian approach
with prior/posterior distributions.
In our approach the probability assigned to each arm is determined by
(normalized) maximum likelihood instead of posterior probability.

This paper is organized as follows.
In Section \ref{section-preliminary}, we give definitions used
throughout this paper and
show the asymptotic bound by \cite{burnetas}, which
is satisfied by any consistent policy.
In Section \ref{section-proposed-policy}, we propose
MED policy and prove that 
it is asymptotically optimal for finite support models.
We also discuss practical implementation issues of 
minimization problem involved in MED. 
In Section \ref{section-experiments}, some simulation results are shown.
We conclude the paper with some remarks in Section \ref{section-remarks}.

\section{Preliminaries}\label{section-preliminary}
In this section we introduce notation of this paper and present
the asymptotic bound 
for a generic model, which is established
 by \cite{burnetas}.

Let $\mf$ be a generic family of probability distributions on
 $\mathbb{R}$ and
let $F_j\in\mf$ be the distribution of $\Pi_j$, $j=1,\dots,K$.
$P_F[\cdot]$ and $\E_F[\cdot]$ denotes the probability and the
expectation under $F\in\mf$, respectively.
When we write e.g. $P_F[X\in A]\;\,(A\subset\mathbb{R})$ or
$\E_F[\theta(X)]\;\,(\theta(\cdot)\mbox{ is a function
}\mathbb{R}\to\mathbb{R})$,
$X$ denotes a random variable with distribution $F$.
We define $F(A)\equiv P_F[X\in A]$ and $\E(F)\equiv \E_F[X]$.

A set of probability distributions for $K$ arms is denoted by
$\bm{F}\equiv (F_1,\dots,F_K)\in \mathcal{F}^K\equiv\prod_{j=1}^K
\mathcal{F}$.
The joint probability and the expected value under 
$\bm{F}$ are  denoted by $P_{\bm F}[\cdot]$, $\E_{\bm F}[\cdot]$, respectively.

The expected value of $\Pi_j$ is denoted by $\mu_j\equiv\E(F_j)$.
We denote the optimal expected value by $\mu^*\equiv \max_{j}\mu_j$.
Let $\jn$ be the arm  chosen in the $n$-th round.
Then 
\begin{eqnarray}
T_j(n)=\sum_{m=1}^n \id[J_m=j], \n
\end{eqnarray}
where $\id[\cdot]$ denotes the indicator function.
For notational convenience we
write 
$T_j'(n)\equiv T_j(n-1),$
which is the number of times the arm $\Pi_j$ has been pulled prior to the
$n$-th round.

Let $\fhatt{j}{t}$ and $\muhatt{j}{t}\equiv \E(\fhatt{j}{t})$  be the
empirical distribution and the mean of the first $t$ rewards from
$\Pi_j$, respectively.
Similarly, let $\fhatn{j}{n}\equiv \fhatt{j}{T_j'(n)}$ and $\muhatn{j}{n}\equiv
\muhatt{j}{T_j'(n)}$ be the empirical distribution and mean of $\Pi_j$
after the first $n-1$ rounds, respectively.
$\hat{\mu}^*(n)\equiv \max_j\muhatn{j}{n}$ denotes the highest
empirical mean after $n-1$ rounds. We call $\Pi_j$ a current best if
$\muhatn{j}{n}=\hat\mu^*(n)$.

Let $\Omega$ denote the whole sample space.  
For an event $A\subset \Omega$, the complement of $A$ is denoted by $A^C$.  
The joint probability of two events $A$ and $B$
under $\bm F$ is written as $P_{\bm F}[A\scap B]$.  For notational
simplicity we often write, e.g.,  
$P_{\bm F}[ \jn = j \scap T_j'(n)=t]$ instead of the more
precise $P_{\bm F}[ \{\jn = j\} \cap \{T_j'(n)=t\}]$.

Finally we define an index for $F\in \mathcal{F}$ and $\mu\in \mathbb{R}$
\begin{eqnarray}
\dinf(F,\mu,\mf)\equiv \inf_{G\in\mathcal{F}:\E(G)>\mu}D(F||G)\n
\end{eqnarray}
where Kullback-Leibler divergence $D(F||G)$ is given by
\begin{eqnarray}
D(F||G)\equiv \begin{cases}
\E_F\left[\log
	       \frac{\mathrm{d}F}{\mathrm{d}G}\right]&\frac{\mathrm{d}F}{\mathrm{d}G}
\mbox{ exists,}\\
+\infty&\mbox{otherwise.}
\end{cases}\n
\end{eqnarray}
$\dinf$ represents how distinguishable $F$ is from
distributions having
expectations larger than $\mu$.
If $\{G\in\mathcal{F}:\E(G)>\mu\}$ is empty, we define
$\dinf(F,\mu,\mf)=+\infty$.
We adopt L\'evy distance
 $L(F,G)$
for distance between two distributions $F,G$.
We use only the fact that the convergence of the L\'evy distance
$L(F,F_n)\to 0$ is equivalent to the weak convergence of $\{F_n\}$ to
distribution $F$
and we write $F_n\to F$ in this sense.

Lai and Robbins \cite{lai} 
 gave  a lower bound for $\E[T_i(n)]$ for any inferior $\Pi_i$
 when a consistent policy is adopted.
However their result was hard to apply for multiparameter models 
and more general non-parametric models.  
Later Burnetas and Katehakis \cite[ahi]{burnetas} extended the
bound to general non-parametric models.  Their bound is given as follows.
\begin{theorem}{\rm \cite[Proposition 1]{burnetas}}
Fix a consistent policy and $\bm{F}\in \mathcal{F}^K$.
If $\E(F_i)<\mu^*$ and $0<\dinf(F_i,\mu^*,\mf)<\infty$,
 then for any $\epsilon>0$ 
\begin{eqnarray}
\lim_{N\to\infty} P_{\bm{F}}\left[T_i(N)\ge \frac{(1-\epsilon)\log
 N}{\dinf(F_i,\mu^*,\mf)}\right]=1.\n
\end{eqnarray}
Consequently
\begin{eqnarray}
\label{eq:liminf-ETj}
\liminf_{N\to\infty} \frac{\E_{\bm{F}}[T_i(N)]}{\log N}
\ge \frac{1}{\dinf(F_i,\mu^*,\mf)}.
\end{eqnarray}
\end{theorem}

\section{Asymptotically Optimal Policy for Finite Support
 Models}\label{section-proposed-policy}
Let $\fb\equiv\{F:|\supp(F)|<\infty,\,\supp(F)\subset [a,b]\}$ 
be the family of distributions
with a {\em finite} bounded support, where $\supp(F)$ is the support
of distribution $F$ and $a,b$ are constants known to the player.
We assume $a=-1,\,b=0$ without loss of generality.
We write $\suppF{F}\equiv\suppp{F}$ and
$\fs{\mx}\equiv\{G\in\fb:\supp(G)\subset\mx\}$ where $\mx$ is an
arbitrary subset of $[-1,0]$.

We consider $\fb$ as a model $\mf$ and propose a policy which we call
the minimum  empirical divergence (MED) policy in this section. 
We prove in Theorem \ref{optMD} that the proposed policy achieves the bound
given in the previous section.
Then, we describe a univariate convex optimization technique to compute
$\dmin$ used in the policy.

Note that the finiteness of the support
can not be determined from
finite samples and every policy for $\fb$ is applicable also for
$\{F:\supp(F)\subset [a,b]\}$.
However our proof of the optimality in this paper
 is for the above $\fb$.
The advantage of assuming the finiteness is that we can employ
the method of types in the large deviation technique.
This enables us to consider all empirical distributions obtained from
each arm.

In this model it is convenient to use
\begin{eqnarray}
\dmin(F,\mu,\fb)\equiv \min_{G\in\fb:\E(G)\ge\mu}D(F||G)\n
\end{eqnarray}
instead of $\dinf(F,\mu,\fb)\equiv \inf_{G\in\fb:\E(G)>\mu}D(F||G)$.
Properties of the minimizer $G^*$ of the right-hand side will be
discussed in Section \ref{section-dinf}. 
\begin{lemma}\label{dmindinf}
$\dmin(F,\mu,\fb)=\dinf(F,\mu,\fb)$ holds for all $F\in \fb$ and $\mu<0$.
\end{lemma}
\begin{proof}
We will prove in Lemma \ref{conti} that $\dmin(F,\mu,\fb)$
 is continuous in $\mu<0$.
$\dmin(F,\mu,\allowbreak \fb)=\dinf(F,\mu,\fb)$ follows easily
 from the continuity.
\end{proof}

\subsection{Optimality of the Minimum Empirical Divergence Policy}
\label{subsec:med}

We now introduce our MED policy.  In MED
an arm is chosen randomly in the following way:

\begin{quote}
{\bf [Minimum Empirical Divergence Policy]}

{\bf Initialization.}\  Pull each arm  once.

{\bf Loop.} For the $n$-th round,
\begin{enumerate}
\item For each $j$ compute
$\hat{D}_j(n)\equiv \dmin(\hat{F}_j(n),\hat{\mu}^*(n),\fb) $.
\item Choose arm $\Pi_j$ according to the
      probability
\begin{eqnarray}
p_j(n)\equiv \frac{\exp(-T_j'(n)\hat{D}_j(n))}{\sum_{i=1}^K
 \exp(-T_i'(n)\hat{D}_i(n))}.
\n
\end{eqnarray}
\end{enumerate}
\end{quote}

Note that 
\begin{eqnarray}
\frac 1 K \le p_j(n)\le 1\label{bound_best}
\end{eqnarray}
for any currently best $\Pi_j$ since $\hat{D}_j(n)=0$. 
As a result, it holds for all $j$ that
\begin{eqnarray}
\frac 1 K \exp(-T_j'(n)\hat{D}_j(n))\le p_j(n)\le
      \exp(-T_j'(n)\hat{D}_j(n)).\label{bound_pj}
\end{eqnarray}

Intuitively, $p_j(n)$ for a currently not best arm $\Pi_j$ corresponds to
the maximum likelihood that $\Pi_j$ is actually the best arm.
Therefore in MED an arm $\Pi_j$
 is pulled with the probability proportional to this
likelihood.

Note that our policy is a randomized policy.  Therefore probability
statements below on MED also involve this randomization.  However for
notational simplicity we omit denoting this randomization.

Now we present the main theorem of this paper.

\begin{theorem}\label{optMD}
Fix $\bm{F}\in\fb^K$ satisfying $\mu_j=\mu^*$ and $\mu_i<\mu^*$
 for all $i\neq j$.
Under MED policy, for any $i\neq j$ and $\ep>0$ it holds that
\begin{eqnarray}
\E_{\bm{F}}[T_i(N)]\le \frac{1+\ep}{\dmin(F_i,\mu^*,\fb)}\log N
+\lo(1).\n
\end{eqnarray}
\end{theorem}

Note that we obtain
$$
\limsup_{N\to\infty}
\frac{\E_{\bm{F}}[T_i(N)]}{\log N} \le \frac{1}{\dmin(F_i,\mu^*,\fb)},
$$
by dividing both sides by $\log N$, letting $N\to\infty$ and finally letting
$\epsilon \downarrow 0$.  In view of (\ref{eq:liminf-ETj}) we see that
MED policy is asymptotically optimal.
We give a proof of Theorem \ref{optMD} in Section \ref{section-proofMD}.

The following corollary shows that the optimality of MED policy
 given in
Theorem \ref{optMD} is a generalization of the optimality in
\cite{burnetas}.
\begin{corollary}
Let $\mx\subset[-1,0]$ be an arbitrary subset of $[-1,0]$
such that $0\in\mathcal{X}$.
Fix $\bm{F}\in\fs{\mx}^K$ satisfying $\mu_j=\mu^*$ and $\mu_i<\mu^*$
 for all $i\neq j$.
Under MED policy, for any $i\neq j$ and $\ep>0$ it holds that
\begin{eqnarray}
\E_{\bm{F}}[T_i(N)]\le \frac{1+\ep}{\dmin(F_i,\mu,\fs{\mx})}\log N
+\lo(1).\label{coro}
\end{eqnarray}
\end{corollary}
\begin{proof}
We prove in Lemma \ref{suppp} that
$\dmin(F,\mu,\fb)=\dmin(F,\mu,\fs{\suppF{F}})$.
On the other hand,
$\dmin(F,\mu,\fs{\suppF{F}})\ge \dmin(F,\mu,\fs{\mx})$
 holds from $\fs{\suppF{F}} \subset \fs{\mx}$.
Then we obtain \eqref{coro} from Theorem \ref{optMD}.
\end{proof}
Note that \eqref{coro} is achieved also by the policy used
in the \cite{burnetas} if $\mathcal{X}$ is fixed and assumed to be
known.
Our result establishes the same bound
 without this assumption.
\subsection{Computation of $\dmin$ and Properties of the Minimizer}
\label{section-dinf}
For implementing MED policy it is essential to efficiently compute
the minimum empirical divergence
$\dmin(\hat{F}_j(n),\allowbreak\hat{\mu}^*(n),\fb)$
for each round. In this subsection, we clarify the nature of the convex
optimization
involved in $\dmin(\hat{F}_j(n),\allowbreak\hat{\mu}^*(n),\fb)$ 
and show how the minimization can be computed efficiently.
In addition, for proofs of Lemma \ref{dmindinf} and Theorem \ref{optMD}, we
need to clarify the behavior of $\dmin(F,\mu,\fb)$ as a function of $\mu$.

First we prove that it is sufficient to consider $\fs{\suppF{F}}$
for the computation of $\dmin(F,\mu,\allowbreak\fb)$:
\begin{lemma}\label{suppp}
$\dmin(F,\mu,\fb)=\dmin(F,\mu,\fs{\suppF{F}})$ holds for any $F\in \fb$.
\end{lemma}
\begin{proof}
Take an arbitrary $G\in\fb\setminus \fs{\suppF{F}}$
 such that $\E(G)\ge \mu$ and $G(\suppF{F})=p<1$.
Define $G'\in\ff{F}$ as
\begin{eqnarray}
G'(\{x\})\equiv \begin{cases}
G(\{0\})+(1-p)&x=0\\
G(\{x\})&x\neq 0,\,x\in\supp(F)\\
0&\mbox{otherwise}.
\end{cases}\n
\end{eqnarray}
Since $D(F||G')\le D(F||G)$ and $\E(G')\ge \E(G)$, we obtain
\begin{eqnarray}
\min_{G\in\fb:\E(G)\ge \mu}D(F||G)\ge
\min_{G'\in\ff{F}:\E(G')\ge \mu}D(F||G').\n
\end{eqnarray}
The converse inequality is obvious from $\ff{F}\subset \fb$.
\end{proof}
In view of this lemma, we simply write $\dmin(F,\mu)$
 instead of
 $\dmin(F,\mu,\fb)=\dmin(F,\allowbreak\mu,\ff{F})$
when the third argument is obvious from the context.

Let $M\equiv |\suppF{F}|$ and
 denote the finite symbols in $\suppF{F}$ by $x_1 \dots, x_M$, i.e.\ 
$\suppp{F}=\{x_1,\dots,\allowbreak x_M\}$.
We assume $x_1=0$ and $x_i<0$ for $i>1$ without loss of generality and
write $f_i\equiv F(\{x_i\})$.

Now the computation of $\dmin(F,\mu)$ is formulated as the following
convex optimization problem for $G=(g_1,\dots,g_M)$ from Lemma \ref{suppp}:
\begin{eqnarray}
\mathrm{minimize}&\quad&\sum_{i=1}^M f_i\log \frac{f_i}{g_i}\nn
\mathrm{subject\;to}&\quad&-g_i\le 0, \ \forall i, \quad
\mu-\sum_{i=1}^M x_i g_i\le 0,\quad \sum_{i=1}^M g_i=1,\label{problem}
\end{eqnarray}
where we define $0\log 0\equiv  0$, $0\log\frac 0 0\equiv  0$,
 and $\frac 1 0\equiv  +\infty$.

It is obvious that $G=F$ is the optimal solution with the optimal value
$0$ when $0\ge\E(F)\ge\mu$.
Also $G=\delta_0$, the unit point mass at 0, is
the unique feasible solution if $\mu=0$.
For $\mu>0$ the problem is infeasible.  
Since these cases are trivial, we
consider the case $\E(F)<\mu<0$ in the following.

Define $h(\nu)$ and its first and second order derivatives as
\begin{eqnarray}
h(\nu)&\equiv& \E_F[\log(1-(X-\mu)\nu)]=
\sum_{i=1}^M f_i\log(1-(x_i-\mu)\nu),
\label{h0}\\
h'(\nu)&\equiv&\hen{}{\nu}h(\nu)=
-\sum_{i=1}^{M}\frac{f_i (x_i-\mu)}{1-(x_i-\mu)\nu}
,\label{h1}\\
h''(\nu)&\equiv&\hen{^2}{\nu^2}h(\nu)=
-\sum_{i=1}^{M}\frac{f_i(x_i-\mu)^2}{(1-(x_i-\mu)\nu)^2}
.\label{h2}
\end{eqnarray}
Now we show in Theorem \ref{thm8} that the computation of $\dmin$ is
expressed as maximization of $h(\nu)$.
Since $h(\nu)$ is concave, it is a univariate convex
optimization problem.
Therefore $\dmin$ can be computed easily by iterative methods such as
Newton's method (see, e.g., \cite{boyd} for general methods of convex
programming).

\begin{theorem}\label{thm8}
Define $\E_F[\mu/X]=\infty$ for the case $F(\{0\})=f_1>0$.
Then following three properties hold for $\E(F)<\mu<0$:

{\rm (i)} $\dmin(F,\mu)$ is written as
\begin{eqnarray}
\dmin(F,\mu)&=&\max_{0\le\nu\le\mom}h(\nu)
\label{unif_opt}
\end{eqnarray}
and the optimal solution $\nu^*\equiv \argmax_{0\le\nu\le\mom}h(\nu)$
 is unique.

In particular for the case $\E[\mu/X]\le 1$, $\nu^*=\momss$ and
 \eqref{unif_opt} is simply written as
\begin{eqnarray}
\dmin(F,\mu)=h(\moms)=\sum_{i=2}^M f_i\log(x_i/\mu).\label{opt_le}
\end{eqnarray}

On the other hand for the case $\E[\mu/X]\ge 1$,
 \eqref{unif_opt} is written as an unconstrained optimization problem
\begin{eqnarray}
\dmin(F,\mu)=
\max_{\nu}h(\nu).
\label{opt_ge}
\end{eqnarray}

{\rm (ii)} $\nu^*$ satisfies
\begin{eqnarray}
 \nu^*\ge\frac{\mu-\E(F)}{-\mu(1+\mu)}.\n
\end{eqnarray}

{\rm (iii)} $\dmin(F,\mu)$ is differentiable in $\mu \in(\E(F),0)$ and
\begin{eqnarray}
\hen{}{\mu}\dmin(F,\mu)=\nu^*.
\n
\end{eqnarray}
\end{theorem}

We give a proof of Theorem \ref{thm8} in Section \ref{section-proofMD}.
\subsection{Proofs of Theorem \ref{optMD} and \ref{thm8}}
\label{section-proofMD}
In this section we give proofs of Theorem \ref{optMD} and \ref{thm8}.
Actually we prove Theorem \ref{optMD} using Theorem \ref{thm8}
and prove Theorem \ref{thm8} independently of Theorem \ref{optMD}.

We first show Lemmas \ref{conti} and \ref{dinf} on properties of
 $\dmin$ to prove Theorem \ref{optMD}.

\begin{lemma}\label{conti}
$\dmin(F,\mu)$ is monotonically increasing in $\mu$ and possesses
 following continuities:
 (1) lower semicontinuous in $F\in\fb$, that is, 
 $\liminf_{F'\to F}\dmin(F',\mu)\ge \dmin(F,\mu)$.
 (2) continuous in $\mu<0$.
\end{lemma}
Note that the continuity in $\mu<0$ is not trivial at $\mu=\E(F)$ because the
differentiability in Theorem \ref{thm8} is valid
only for the case $\E(F)<\mu<0$ and $\dmin(F,\mu)$ may not be
 differentiable at $\mu=\E(F)$.
\begin{proof}
The monotonicity is obvious from the definition of $\dmin$.

(1)
Fix an arbitrary $\ep>0$.
From \eqref{unif_opt} and the continuity of $h(\nu)$, 
there exists $\nu_0\in[0,\momss)$ such that
 $\E_F[\log(1-(X-\mu)\nu_0)]\ge \dmin(F,\mu)-\ep$.
Then we obtain
\begin{eqnarray}
\liminf_{F'\to F}\dmin(F',\mu)&\ge&
\liminf_{F'\to F} \E_{F'}[\log(1-(X-\mu)\nu_0)]\nn
&=&\E_F[\log(1-(X-\mu)\nu_0)]\label{levy_conv}\\
&\ge&\dmin(F,\mu)-\ep.\n
\end{eqnarray}
Note that $\log(1-(x-\mu)\nu_0)$ is continuous
 and bounded in $x\in[-1,0]$ and
\eqref{levy_conv} follows from the definition of weak convergence.
The lower semicontinuity holds since $\ep$ is arbitrary.

(2) The continuity is obvious for $\mu>\E(F)$ from the differentiability in
 Theorem \ref{thm8}.
The case $\mu<\E(F)$ is also obvious since $\dmin(F,\mu)=0$ holds for $\mu\le
 \E(F)$.
Then it is sufficient to show
\begin{eqnarray}
\lim_{\mu\downarrow
 \E(F)}\dmin(F,\mu)=\dmin(F,\E(F))=0.\label{conti_mu}
\end{eqnarray}
From \eqref{unif_opt} and the concavity of $h(\nu)$, it holds that
\begin{eqnarray}
&&h(0)\le \dmin(F,\mu)\le h(0)+h'(0)\moms\nn
 &\Leftrightarrow&\quad\:\, 0\le \dmin(F,\mu)\le
 \mbox{$\frac{\E(F)}{\mu}$}-1
\n
\end{eqnarray}
for $\mu>\E(F)$.
\eqref{conti_mu} is obtained by letting $\mu\downarrow \E(F)$.
\end{proof}

\begin{lemma}\label{dinf}
Fix arbitrary $\mu,\mu'\in(-1,0)$ satisfying $\mu'<\mu$.
Then there exists $C(\mu,\mu')>0$ such that
\begin{eqnarray}
\dmin(F,\mu)-\dmin(F,\mu')
&\ge&
C(\mu,\mu').\n
\end{eqnarray}
for all $F\in \fb$ satisfying $\E(F)<\mu'$.
\end{lemma}
\begin{proof}
Since $\dmin(F,\mu)$ is differentiable in $\mu>\E(F)$
 from Theorem \ref{thm8},
we have
\begin{eqnarray}
\dmin(F,\mu)-\dmin(F,\mu')
&=&\int_{\mu'}^{\mu}\hen{}{u}\dmin(F,u)\rd
			       u\nn
&\ge&\int_{\mu'}^{\mu}
\frac{u-\mu'}{-u(1+u)} \rd u\nn
&\ge&
\int_{\mu'}^{\mu}
\frac{u-\mu'}{-\mu'(1+\mu)} \rd u\nn
&=&
\frac{(\mu-\mu')^2}{-2\mu'(1+\mu)}
\;\:\big(=:C(\mu,\mu')\big).\n
\end{eqnarray}
\end{proof}

\begin{proof}[Proof of Theorem \ref{optMD}]
We define more notation used in the following proof.
We fix $j=1$ and let $L\equiv \{2,\dots,K\}$.
Then, $\mu^*=\mu_1$ and $\mu_k<\mu_1$ for $k\in L$.
For notational convenience  we denote
$J_n(i)\equiv\{\jn=i\}$
which is the event that the arm $\Pi_i$ is
 pulled at the $n$-th round.

We  simply write $\E[\cdot], P[\cdot]$ as an expectation and a probability 
under $\bm{F}$ and the randomization in the policy. 
Now we define events $A_n,\,B_n,\,C_n,\,D_n$ as follows:
\begin{eqnarray*}
A_n&\equiv& \left\{ \hat{D}_i(n)\ge \frac{\di{i}}{1+\ep/2} \right\} \\
B_n&\equiv&\{ \hat{\mu}_1(n)\ge \mu_1-\de \}\\
C_n&\equiv&\{ \hat{\mu}_1(n)<\mu_1-\de \scap \max_{k\in L}\hat{\mu}_k(n)<\mu_1-\de \}\\
D_n&\equiv&\{ \hat{\mu}_1(n)<\mu_1-\de \scap \max_{k\in L}\hat{\mu}_k(n)\ge\mu_1-\de\}
\end{eqnarray*}
where $\de>0$ is a constant satisfying
 $\max_{k\in L}\mu_k < \mu_1-\de$
 which is set
sufficiently small in the evaluation on $B_n$.
Note that
$B_n \cup C_n \cup D_n = \Omega$
and 
each $\id[J_n(i)]$ in the sum
$T_i(N)=\sum_{n=1}^N \id[J_n(i)]$ 
is bounded from above by
\begin{eqnarray}
\id[J_n(i)]&\le& 
\id[J_n(i)\cap A_n]
+ \id[J_n(i)\cap C_n]
+ \id[J_n(i)\cap A^C_n \cap B_n]
+ \id[J_n(i)\scap D_n].\quad
\label{eq:five-terms}
\end{eqnarray}

In the following Lemmas \ref{lemA}-\ref{lemD} we  bound the expected values of sums of the four
terms on the right-hand side of (\ref{eq:five-terms})
in this order
 and they are sufficient to prove Theorem \ref{optMD}.
\end{proof}

\begin{lemma}\label{lemA}
Fix an arbitrary $\ep>0$.
Then it holds that
\begin{eqnarray}
\E\left[\sum_{n=1}^N \id[J_n(i)\cap A_n]\right]
\le \frac{1+\ep}{\dmin(F_i,\mu^*)}\log N
+\so(1).\n
\end{eqnarray}
\end{lemma}

\begin{lemma}\label{lemC}
\begin{eqnarray}
\E\left[\sum_{n=1}^N \id[J_n(i)\cap C_n]\right]=\lo(1).\n
\end{eqnarray}
\end{lemma}

\begin{lemma}\label{lemnAB}
\begin{eqnarray}
\E\left[\sum_{n=1}^N \id[J_n(i)\cap A^C_n \cap B_n]\right]=\lo(1).\n
\end{eqnarray}
\end{lemma}

\begin{lemma}\label{lemD}
\begin{eqnarray}
\E\left[ \sum_{n=1}^N \id[J_n(i)\cap D_n]\right]=\lo(1).\n
\end{eqnarray}
\end{lemma}

Before proving these lemmas,
we give intuitive interpretations for these terms.

$A_n$ represents the event that the estimator
 $\hat{D}_i(n)=\dmin(\fhatn{i}{n},\muhat^*(n))$
 of $\dmin(F_i,\\ \mu^*)$ is
 already close to $\dmin(F_i,\mu^*)$ and
 $\Pi_i$ is pulled with a small probability.
After sufficiently many rounds $A_n$ holds with
 probability close to $1$ and the term
$\sum_{n=1}^N \id[J_n(i)\cap A_n]$ is the main term of $T_i(N)$.

Other terms of \eqref{eq:five-terms} represent events
that $\Pi_i$ is pulled when each estimator is not yet close to the
true value.
The term involving $C_n$ is essential for the consistency of MED.

$A_n^C\scap B_n$ represents the following event:
$\hat{D}_i(n)$ has not converged because
$\fhatn{i}{n}$ is not close to $F_i$ although $\hat{\mu}^*(n)$ is
already close to $\mu_1$.
In this event $\Pi_i$ is pulled and therefore $\fhatn{i}{n}$ is
updated more frequently.
 As a result, $A_n^C\scap B_n$ happens only for a few $n$.

Similarly, $D_n$ represents the event that $\hat{\mu}_k$ happens to be
large for some $k\in L$.
Also in this event $\fhatn{k}{n}$ is updated more frequently 
 and $D_n$ happens only for a few $n$.

On the other hand, $C_n$ represents the event that $\hat{\mu}_1$ is not
yet close to $\mu_1$.
It requires many rounds for $\Pi_1$ to be pulled since $\Pi_1$ seems to
be inferior in this event.
Therefore $C_n$ may happen for many $n$.

\begin{proof}[Proof of Lemma \ref{lemA}]
By partitioning $\id[J_n(i)\cap A_n]$ according to the number of occurrences
 $\allowbreak \sum_{m=1}^{n-1}\id[J_m(i)\cap A_m]$ of the event
 $J_m(i)\cap A_m$ before the $n$-th round, we have
\begin{eqnarray}
\lefteqn{
\sum_{n=1}^N \id[\jn(i) \cap A_n]
}\quad\nn
&\le&
\frac{(1+\epsilon)\log N}{\di{i}}
+\sum_{n=1}^N \id\left[\jn(i)\scap A_n \scap
\left\{\sum_{m=1}^{n-1} \id[J_m(i) \cap A_m]> 
\frac{(1+\epsilon)\log N}{\di{i}}\right\}
\right].\n
\end{eqnarray}
Since $\sum_{m=1}^{n-1}\id[J_m(i)\scap A_m]\le
 \sum_{m=1}^{n-1}\id[J_m(i)]=T'_i(n)$, we obtain
\begin{eqnarray}
\sum_{n=1}^N \id[\jn(i) \cap A_n]
\le
\frac{(1+\epsilon)\log N}{\di{i}}+
\sum_{n=1}^N \id\left[\jn(i)\scap A_n \scap T'_i(n)>
\frac{(1+\epsilon)\log N}{\di{i}}
\right].\n
\end{eqnarray}
Taking the expected value we have
\begin{eqnarray}
\lefteqn{
\E\left[\sum_{n=1}^N \id[\jn(i) \cap A_n]\right]
}\quad\nn
&\le&
\frac{(1+\epsilon)\log N}{\di{i}}+
\sum_{n=1}^N
P\left[\jn(i)\scap A_n \scap T'(n)>\frac{(1+\epsilon)\log N}{\di{i}}
\right].\nn
&\le&
\frac{(1+\epsilon)\log N}{\di{i}}+ 
 \sum_{n=1}^N  P\left[J_n(i) \,\bigg|\, A_n\scap T_i'(n)>\frac{(1+\epsilon)\log
 N}{\di{i}}\right] \nn
&\le& \frac{(1+\epsilon)\log N}{\di{i}} +
N\exp\left(-\frac{(1+\epsilon)\log N}{\di{i}}
\frac{\di{i}}{1+\ep/2}\right)
\qquad(\mbox{by \eqref{bound_pj}})\nn
&=&\frac{(1+\epsilon)\log N}{\di{i}} +
 N^{-\frac{1+\ep}{1+\ep/2}+1}\n
\end{eqnarray}
The lemma is proved since $N^{-\frac{1+\ep}{1+\ep/2}+1}=\so(1)$.
\end{proof}

\begin{proof}[Proof of Lemma \ref{lemC}]
First we have
\begin{eqnarray}
\sum_{n=1}^N \id[J_n(i)\cap C_n]
&\le&
\sum_{n=1}^N \id[\jn\in L\scap C_n]\nn
&\le&
\sum_{t=1}^N\sum_{n=1}^{\infty}\id[\jn\in L\scap T_1'(n)=t\scap
C_n].\label{7-0}
\end{eqnarray}
From the technique of type \cite[Lemma 2.1.9]{LDP},
it holds for any type $Q\in \fb$ that
\begin{eqnarray}
P_{F_1}[\fhatt{1}{t}=Q]\le \exp(-tD(Q||F_1))
\le\exp(-t\dmin(Q,\mu_1)).\label{7-19}
\end{eqnarray}

Let $\bm{R}=(R_1,\dots,R_m)$ be the smallest $m$ integers in
 $\{n:T_1'(n)=t\scap C_n\}$. $\bm{R}$ is well defined
on the event
$m\le \sum_{n=1}^{\infty}\id[\jn\in L\scap T_1'(n)=t\scap C_n]$.
Let $\bm{r}=(r_1,\dots,r_m)\in\mathbb{N}^m$ be a realization of $\bm{R}$.
Here recall that we write an event e.g.
 ``$\cdots \scap \bm{R}=\bm{r} \scap \fhatt{1}{t}=Q$'' instead of 
``$\cdots \scap \{\bm{R}=\bm{r}\} \scap \{\fhatt{1}{t}=Q\}$''.
Then we obtain for any $\bm{r}$ that
{\allowdisplaybreaks
\begin{eqnarray}
\lefteqn{
P\left[
\left\{\sum_{n=1}^{\infty}\id[\jn\in L\scap T_1'(n)=t\scap C_n]\ge m
\right\}
\scap
\bm{R}=\bm{r}\scap \fhatt{1}{t}=Q
\right]
}\quad\nn
&=&
P\left[\bigcap_{l=1}^m\{J_{r_l}\in L\}\scap \bm{R}=\bm{r}\scap
  \fhatt{1}{t}=Q\right]\nn
&=&
P_{F_1}[\fhatt{1}{t}=Q]
\prod_{l=1}^m \Bigg(P\left[R_l=r_l
\ \Big|\ 
\bigcap_{k=1}^{l-1} \{J_{r_k}\in L\scap R_k=r_k\}\scap
\fhatt{1}{t}=Q
	  \right]\nn
&&\qquad\times  P\left[J_{r_l}\in L
\ \Big|\ 
R_l=r_l\scap\bigcap_{k=1}^{l-1} \{J_{r_k}\in L\scap R_k=r_k\}\scap
\fhatt{1}{t}=Q
\right]\Bigg)\nn
&\le&
P_{F_1}[\fhatt{1}{t}=Q]
\prod_{l=1}^m \Bigg( P\left[R_l=r_l
\ \Big|\ 
\bigcap_{k=1}^{l-1} \{J_{r_k}\in L\scap R_k=r_k\}\scap
\fhatt{1}{t}=Q\right]\nn
&&\qquad\times
\left(1-\frac 1 K \exp(-t
			\dmin(Q,\,\mu_1-\de))\right)\Bigg)
\qquad(\mbox{by \eqref{bound_pj} and $\hat{\mu}^*(R_l)<\mu_1-\de$})\nn
&=&
P_{F_1}[\fhatt{1}{t}=Q]
\left(1-\frac 1 K \exp(-t
			\dmin(Q,\,\mu_1-\de))\right)^m\nn
&&\qquad\times\prod_{l=1}^m P\left[R_l=r_l
\ \Big|\ 
\bigcap_{k=1}^{l-1} \{J_{r_k}\in L\scap R_k\in r_l\}\scap
\fhatt{1}{t}=Q
\right].\n
\end{eqnarray}
}
By taking the disjoint union of $\bm{r}$, we have
\begin{eqnarray}
\lefteqn{
P\left[\left\{\sum_{n=1}^{\infty}\id[\jn\in L\scap T_1'(n)=t\scap C_n]\ge m
\right\}
\scap \fhatt{1}{t}=Q \right]
}\nn
&\qquad\le&
P_{F_1}[\fhatt{1}{t}=Q]
\left(1-\frac 1 K \exp(-t
			\dmin(Q,\,\mu_1-\de))\right)^m.
\label{7-20}
\end{eqnarray}
Then we have
\begin{eqnarray}
\lefteqn{
\E\Bigg[\sum_{n=1}^{\infty}\id[\jn\in L\scap T_1'(n)=t\scap C_n]
\Bigg]
}\quad\nn
&=&
\sum_{Q:\E(Q)<\mu_1-\de}\sum_{m=1}^{\infty}
P\Bigg[
\left\{\sum_{n=1}^{\infty}\id[\jn\in L\scap T_1'(n)=t\scap C_n]\ge
 m\right\}
\scap \fhatt{1}{t}=Q\Bigg]\nn
&\le&
\sum_{Q:\E(Q)<\mu_1-\de}\sum_{m=1}^{\infty}
\exp(-t\dmin(Q,\mu_1))
\left(1-\frac 1 K \exp(-t
			\dmin(Q,\,\mu_1-\de))\right)^m\nn
&&\phantom{wwwwwwwwwwwwwwwwwwwwwwwwwwwwwwww}
(\mbox{by \eqref{7-19} and \eqref{7-20}})\nn
&\le&
K\sum_{Q:\E(Q)<\mu_1-\de}
\exp\Big(-t\big(\dmin(Q,\mu_1)-\dmin(Q,\,\mu_1-\de)\big)\Big)\nn
&&
\nn
&\le&
K\sum_{Q:\E(Q)<\mu_1-\de}
\exp(-t\,C(\mu_1,\mu_1-\de))
\qquad(\mbox{by Lemma \ref{dinf}})\nn
&\le&
K(t+1)^{|\supp(F_1)|}\exp(-t\,C(\mu_1,\mu_1-\de)).\label{7-9}
\end{eqnarray}
The last inequality holds since there are at most
 $(t+1)^{|\supp(F_1)|}$ combinations as a type
 of $t$ samples from $F_1$.

Finally we obtain from \eqref{7-0}, \eqref{7-9} and $C(\mu_1,\mu_1-\de)>0$ that
\begin{eqnarray}
\E\left[\sum_{n=1}^N \id[J_n(i)\cap C_n]\right]
\le
\sum_{t=1}^N K(t+1)^{|\supp(F_1)|} \exp(-tC(\mu_1,\mu_1-\de))
=\lo(1) \n
\end{eqnarray}
and the proof is completed.
\end{proof}

In the proofs of remaining two lemmas, 
we use \cite[Theorem 6.2.10]{LDP} on the empirical distribution: 
\begin{theorem}[Sanov's Theorem]\label{sanov}
For every closed set $\Gamma$ of probability distributions
\begin{eqnarray}
\limsup_{t\to\infty}\frac 1 t \log P_F[\hat{F}_t\in\Gamma]
\le
-\inf_{G\in \Gamma}D(G||F).\n
\end{eqnarray}
where $\hat{F}_t$ is the empirical distribution of $t$ samples from $F$.
\end{theorem} 

\begin{proof}[Proof of Lemma \ref{lemnAB}]
We apply Sanov's Theorem with $F=F_i$ and
\begin{eqnarray}
\Gamma=\{G\in\fb :L(F_i,G)\ge \de_1\}\n
\end{eqnarray}
where $\de_1>0$ is a constant.
Since 
$
\inf_{G\in\Gamma}D(G||F_i)>0,
$
there exists a constant $C_1>0$ such that
\begin{eqnarray}
P_{F_i}[\fhatt{i}{t}\in\Gamma] \le \exp(-C_1t)\label{pgamma}
\end{eqnarray}
 for sufficiently large $t$.

Now we show
\begin{eqnarray}
\{A^C_n\scap B_n\}\subset \{\hat{F}_i(n)\in\Gamma\}
\label{8-26}
\end{eqnarray}
or equivalently
 $\{\fhatn{i}{n}\notin \Gamma \scap B_n\}\subset A_n$
for sufficiently small $\de_1$.
If $\hat{F}_i(n)\notin \Gamma_1$ and $B_n$, then
\begin{eqnarray} 
\dmin(\fhatn{i}{n},\hat{\mu}^*(n))
\ge\dmin(\fhatn{i}{n},\mu^*-\de)\n
\end{eqnarray}
from $\hat{\mu}^*(n)\ge\hat{\mu}_1(n)\ge\mu_1-\de=\mu^*-\de$
 and the monotonicity of $\dmin$ in $\mu$.
Since $\dmin(F_i,\mu^*-\de)>0$, for sufficiently small $\de_1$ we obtain 
\begin{eqnarray}
\dmin(\fhatn{i}{n},\,\mu^*-\de)\ge\frac{\dmin(F_i,\mu^*-\de)}{1+\ep/3}\n
\end{eqnarray}
from the lower semicontinuity in $F$ of $\dmin$ in
 Lemma \ref{conti}.
Moreover, from the continuity of $\dmin$ in $\mu$, it holds for sufficiently
 small $\de$ that
\begin{eqnarray}
\frac{\dmin(F_i,\mu^*-\de)}{1+\ep/3}\ge\frac{\di{i}}{1+\ep/2}.\n
\end{eqnarray}
Then $A_n$ holds and \eqref{8-26} is proved.

From \eqref{8-26} we obtain
\begin{eqnarray}
\lefteqn{
\Bigg\{\sum_{n=1}^{N}\id[\jn(i)\scap A^C_n\scap B_n]\ge m\Bigg\}
}\nn
&\qquad\subset&
\Bigg\{\sum_{n=1}^N\id[\jn(i)\scap\fhatn{i}{n}\in\Gamma]\ge m\Bigg\}\nn
&\qquad=&
\Bigg\{\sum_{t=1}^N\id\left[\bigcup_{n=1}^N
\left\{\jn(i)\scap T'_i(n)=t\scap\fhatt{i}{t}\in\Gamma\right\}
\right]\ge m\Bigg\}\label{henkei}\\
&\qquad\subset&
\Bigg\{\sum_{t=1}^N\id\left[\fhatt{i}{t}\in\Gamma
\right]\ge m\Bigg\}\nn
&\qquad\subset&
\bigcup_{l=m}^N \{\fhatt{i}{l}\in\Gamma\}.\label{pab23}
\end{eqnarray}
\eqref{henkei} follows because there is at most one $n$ 
such that $J_n(i)\scap T_i(n)=t$.

Finally, from \eqref{pgamma} and \eqref{pab23} we obtain
\begin{eqnarray}
\E\left[\sum_{n=1}^N \id[J_n(i)\cap A^C_n \cap B_n]\right]
&=&
\sum_{m=1}^N P\left[\sum_{n=1}^{N}\id[\jn(i)\scap A^C_n\scap B_n]\ge
m\right]\nn
&\le&
\sum_{m=1}^N\sum_{l=m}^N P_{F_i}[\fhatt{i}{l}\in\Gamma]\nn
&=&
\lo(1).\n
\end{eqnarray}
\end{proof}

\begin{proof}[Proof of Lemma \ref{lemD}]
First we simply bound $\sum_{n=1}^N \id[J_n(i)\cap D_n]$ by
$$
\sum_{n=1}^N \id[J_n(i)\cap D_n] \le \sum_{n=1}^\infty\id[D_n].
$$
Since
$
D_n\subset\bigcup_{k\in L}\{\muhatn{k}{n}=\hat{\mu}^*(n)>\mu_1-\de\}, 
$
it holds that
\begin{eqnarray}
\sum_{n=1}^\infty \id[D_n]
&\le&
\sum_{k\in L}\sum_{n=1}^\infty
\id[\muhatn{k}{n}=\hat{\mu}^*(n)>\mu_1-\de]\nn
&=&
\sum_{k\in L}\sum_{t=1}^{\infty}\sum_{n=1}^{\infty}
\id[\muhatt{k}{t}=\hat{\mu}^*(n)>\mu_1-\de\scap T'_k(n)=t].\label{dn1}
\end{eqnarray}

Now we use a reasoning similar to \eqref{7-20}.
Let $\bm{R}=(R_1,\dots,R_m)$ be the smallest $m$ integers in
 $\{n:T_k'(n)=t\scap \muhatt{k}{t}=\hat{\mu}^*(n)
>\mu_1-\de\}$.
 $\bm{R}$ is well defined
on the event
$m\le \sum_{n=1}^\infty
\id[T_k'(n)=t\scap \muhatt{k}{t}=\hat{\mu}^*(n)>\mu_1-\de]$.
Then we have
\begin{eqnarray}
\lefteqn{
P\left[\sum_{n=1}^{\infty}
\id[T_k'(n)=t\scap \muhatt{k}{t}=\hat{\mu}^*(n)>\mu_1-\de]\ge m\right]
}\nn
&\qquad=&
P_{F_k}[\muhatt{k}{t}>\mu_1-\de]\,P\left[\sum_{n=1}^{\infty}
\id[T_k'(n)=t\scap \muhatt{k}{t}=\hat{\mu}^*(n)]\ge m \,\Bigg|\,
\muhatt{k}{t}>\mu_1-\de
\right]\nn
&\qquad\le&
P_{F_k}[\muhatt{k}{t}>\mu_1-\de]\,P\left[
\prod_{l=1}^{m-1}\left\{
J_{R_l}\neq k
\right\}\,\Bigg|\,
\muhatt{k}{t}>\mu_1-\de
\right]\nn
&\qquad\le&
P_{F_k}[\muhatt{k}{t}>\mu_1-\de]\left(1-\frac{1}{K}\right)^{m-1}\n
\end{eqnarray}
from $\muhatn{k}{R_l}=\hat{\mu}^*(R_l)$ and \eqref{bound_best}.
Therefore we obtain
\begin{eqnarray}
\lefteqn{
\E\left[\sum_{n=1}^{\infty}
\id[\muhatt{k}{t}=\hat{\mu}^*(n)>\mu_1-\de\scap T'_k(n)=t]
\right]
}\nn
&\qquad=&
\sum_{m=1}^{\infty}P\left[\sum_{n=1}^{\infty}
\id[T_1'(n)=t\scap \muhatt{k}{t}=\hat{\mu}^*(n)>\mu_1-\de]\ge m\right]\nn
&\qquad\le&
K\,P_{F_k}[\muhatt{k}{t}>\mu_1-\de].\label{dn2}
\end{eqnarray}

On the other hand, it holds from Sanov's theorem that for a constant $C_2>0$
\begin{eqnarray}
P_{F_k}[\muhatt{k}{t}>\mu_1-\de]=\lo(\exp(-C_2t))\label{lem9LDP}
\end{eqnarray}
by setting $F=F_k$ and $\Gamma=\{G\in\fb:\E(G)\ge\mu_1-\de\}$.
From \eqref{dn1}, \eqref{dn2} and \eqref{lem9LDP}, we obtain
\begin{eqnarray}
\E\left[\sum_{n=1}^N\id[D_n]\right]&\le&
\sum_{k\in L}\sum_{t=1}^\infty K\lo(\exp(-C_2 t))\nn
&=&\lo(1).\n
\end{eqnarray}
\end{proof}

\begin{proof}[Proof of Theorem \ref{thm8}]
(i)
$h''(\nu)=0$ holds only for the degenerate case that $f_i=1$ at $x_i=\mu$
and this case does not satisfy the assumption $\E(F)<\mu$.
Therefore $h''(\nu)<0$ and $h(\nu)$ is strictly concave.
$\nu^*$ is unique from the strict concavity.

Now we show \eqref{unif_opt}, \eqref{opt_le} and \eqref{opt_ge}
 by the technique of Lagrange multipliers.
The Lagrangian function for \eqref{problem} is written as
\begin{eqnarray}
\sum_{i=1}^M f_i\log \frac{f_i}{g_i}-\sum_{i=1}^M \lambda_i
 g_i+\nu\left(\mu-\sum_{i=1}^M x_i g_i\right)+\xi\sum_{i=1}^M g_i.\n
\end{eqnarray}
Then there exists a Kuhn-Tucker vector
 $(\lambda_1^*,\cdots,\lambda_M^*,\nu^*,\xi^*)$ for the problem
 \eqref{problem} from \cite[Theorem 28.2]{rockafellar}. 
On the other hand it is obvious that the problem \eqref{problem} has an
 optimal solution $G^*=(g_1^*,\cdots,g_M^*)$.
From \cite[Theorem 28.3]{rockafellar},
 $(g_1^*,\cdots,g_M^*)$ is an optimal value
 and
 $(\lambda_1^*,\cdots,\lambda_M^*,\nu^*,\xi^*)$ is a Kuhn-Tucker vector
if and only if
the following Kuhn-Tucker
 conditions are satisfied:
\begin{eqnarray}
&&-\frac{f_i}{g_i^*}-\lambda_i^*-x_i\nu^*+\xi^*=0, \ \forall i\nn
&&g_i^*\ge 0,\,\lambda_i\ge 0,\,g_i\lambda_i=0,\,\forall i,\nn
&&\sum_{i=1}^M x_i g_i^*\ge\mu,\,\nu^*\ge 0,\,\nu^*\left(\mu-\sum_{i=1}^M
						  x_ig_i^*\right)=0,\nn
&&\sum_{i=1}^M g_i^*=1.\n
\end{eqnarray}

First we consider the case $\E_F[\mu/X]\le 1$.
In this case, it is easily checked that
\begin{eqnarray}
g_i^*&=&\begin{cases}
\frac{\mu f_i}{x_i}&i\neq 1\\
1-\sum_{i=2}^{M}\frac{\mu f_i}{x_i}&i=1,
\end{cases}\n
\end{eqnarray}
$\lambda_i^*=0$, $\nu^*=\momss$ and $\xi^*=0$ satisfy Kuhn-Tucker
 conditions since $f_1=0$ and $f_i>0$ for $i\neq 1$.
Therefore \eqref{opt_le} is obtained.
\eqref{unif_opt} follows from $h'(\momss)\ge 0$ and the concavity of
 $h(\nu)$.

Now we consider the second case $\E_F[\mu/X]\ge 1$.
Since $h'(0)>0$, $h'(\momss)\le 0$ and $h(\nu)$ is concave,
\begin{eqnarray}
\max_{0\le\nu\le\mom}h(\nu)=\max_{\nu}h(\nu)\label{const_iranai}
\end{eqnarray}
holds and $\nu^*=\argmax_{0\le\nu\le\momss}h(\nu)$
satisfies
\begin{eqnarray}
-h'(\nu^*)=\sum_{i=1}^{M}f_i\frac{
 x_i-\mu}{1-(x_i-\mu)\nu^*}=0.\label{bibun0}
\end{eqnarray}
From \eqref{bibun0} we obtain
\begin{eqnarray}
\sum_{i=1}^{M}\frac{f_i}{1-(x_i-\mu)\nu^*}
&=&
\sum_{i=1}^{M}f_i\frac{1-(x_i-\mu)\nu^*}{1-(x_i-\mu)\nu^*}
+\nu^*\sum_{i=1}^{M}f_i\frac{x_i-\mu}{1-(x_i-\mu)\nu^*}=1
\label{nu_sum1}
\end{eqnarray}
and
\begin{eqnarray}
\sum_{i=1}^{M}\frac{f_i x_i}{1-(x_i-\mu)\nu^*}
&=&
\sum_{i=1}^{M}f_i\frac{x_i-\mu}{1-(x_i-\mu)\nu^*}+
\mu\sum_{i=1}^{M}\frac{f_i}{1-(x_i-\mu)\nu^*}=\mu.\label{nu_ave}
\end{eqnarray}
From \eqref{nu_sum1} and \eqref{nu_ave},
it is easily checked that
\begin{eqnarray}
g_i^*&=&\begin{cases}
\frac{f_i}{1-(x_i-\mu)\nu^*}&f_i>0\\
0&f_i=0,
\end{cases}\nn
\lambda_i^*&=&\begin{cases}
0&f_i>0\\
1-(x_i-\mu)\nu^*&f_i=0,
\end{cases}\n
\end{eqnarray}
$\xi^*=1+\mu\nu^*$
and $\nu^*$
 satisfy Kuhn-Tucker conditions and \eqref{unif_opt} is obtained.
\eqref{opt_ge} follows immediately from \eqref{const_iranai}.

(ii)
The claim is obviously true for the case $\E_F[\mu/X]\le 1$
 and we consider the case
 $\E_F[\mu/X]\ge 1$.

Define
\begin{eqnarray}
w(x,\nu)\equiv\frac{x-\mu}{1-(x-\mu)\nu}.\n
\end{eqnarray}
For any fixed $\nu\in[0,\momss]$, $w(x,\nu)$ is convex in $x\in[-1,0]$.
Therefore
\begin{eqnarray}
h'(\nu)
&=&
-\sum_{i=1}^M f_i w(x_i,\nu)\nn
&\ge&
-\sum_{i=1}^Mf_i \big(-x_i w(-1,\nu)+(1+x_i)w(0,\nu)\big)\nn
&=&
\E(F)w(-1,\nu)-(1+\E(F))w(0,\nu).\label{totu}
\end{eqnarray}
The right-hand side of \eqref{totu} is $0$
 for $\nu=\numins$
 and therefore
\begin{eqnarray}
h'\left(\numin\right)\ge 0.\n
\end{eqnarray}
Since $h'(\nu)$ is monotonically decreasing,
$\nu^*\ge \numins$
is proved. 

(iii)
It is obvious
that $\hen{}{\mu}\dmin(F,\nu)=\nu^*=\momss$ 
for
 $\E_F[\mu/X]<1$ and 
\begin{eqnarray}
\lim_{\ep\downarrow
 0}\frac{\dmin(F,\mu+\ep)-\dmin(F,\mu)}{\ep}=\mom\n
\end{eqnarray}
 for $\E_F[\mu/X]=1$.

Define $\dmin'(F,\mu)\equiv \max_{\nu}h(\nu)$.
Then $\dmin(F,\mu)=\dmin'(F,\mu)$ for the case $\E_F[\mu/X]\ge 1$. 
From \cite[Corollary 3.4.3]{fiacco},
$\dmin'(F,\mu)$ is differentiable in $\mu$ with
\begin{eqnarray}
\hen{}{\mu}\dmin'(F,\nu)=\hen{}{\mu}h(\nu)\bigg|_{\nu=\nu^*}=\nu^*.\n
\end{eqnarray}
Therefore we obtain
\begin{eqnarray}
\hen{}{\mu}\dmin(F,\nu)=\hen{}{\mu}\dmin'(F,\nu)=\nu^*\n
\end{eqnarray}
for $\E_F[\mu/X]> 1$ and
\begin{eqnarray}
\lim_{\ep\downarrow
 0}\frac{\dmin(F,\mu-\ep)-\dmin(F,\mu)}{-\ep}
&=&
\lim_{\ep\downarrow
 0}\frac{\dmin'(F,\mu-\ep)-\dmin'(F,\mu)}{-\ep}\nn
&=&\nu^*=\mom\n
\end{eqnarray}
 for $\E_F[\mu/X]=1$.
\end{proof}

\section{Experiments}
\label{section-experiments}
In this section, we present some simulation results on our MED
and UCB policies in \cite{ucb}.

First we give an algorithm for computing $\nu^*$ and $\dmin(F,\mu)$
 with parameters $r,\nu_0$, which we denote by $\dmin(F,\mu;r,\nu_0)$.
Here $r$ is a repetition number and $\nu_0$ is an initial value of $\nu$ for
the optimization in Theorem \ref{thm8}.
Recall that $h,h',h''$ are defined in \eqref{h0}, \eqref{h1} and \eqref{h2}.
\mbox{}\\
\noindent {\bf [Computation of $\dmin(F,\mu; r,\nu_0)$]}
\begin{algorithmic}
\REQUIRE $r>0$, $\nu_0\ge 0$;
\IF {$f_1=0$ and $\mu\sum_{i\neq 1}\frac{f_i}{x_i}\le 1$}
 \RETURN $
\left(h\left(\mom\right),\mom\right)$;
\ENDIF
\STATE $\unu,\nu:=\frac{\mu-\E(F)}{-\mu(1+\mu)};\,
\onu:=\mom$;
\IF {$\nu_0\in(\unu,\onu)$}
 \STATE $\nu:=\nu_0$;
\ENDIF
\FOR {$t:=1$ to $r$}
 \IF{$h'(\nu)>0$}
  \STATE $\unu:=\nu$;
 \ELSE
  \STATE $\onu:=\nu$;
 \ENDIF
 \STATE $\nu:=\nu-h'(\nu)/h''(\nu)$;
 \IF {$\nu\notin (\unu,\onu)$}
  \STATE $\nu:=\frac{\unu+\onu}{2}$;
 \ENDIF
\ENDFOR
\RETURN
 $
\left(\max_{\nu'\in\{\unu,\onu,\nu\}}h(\nu'),\,
\argmax_{\nu'\in\{\unu,\onu,\nu\}}h(\nu')\right)$;
\end{algorithmic}
\vspace{1.5mm}
In this algorithm, a lower and an upper bound of $\nu^*$ are
 given by $\unu$ and $\onu$, respectively.
In each step, the next point is determined based on Newton's method
by $\nu:=\nu-h'(\nu)/h''(\nu)$.
When $\nu$ does not improve the bounds $\unu,\,\onu$,
the next point is determined by bisection method, $\nu:=(\unu+\onu)/2$.
The complexity of the algorithm is given by $\lo(r\,|\supp(F)|)$.

The complexity $\lo(r\,|\supp(F)|)$ is not very small
 when $|\supp(F)|$ is large.
Especially it requires
$\lo(r\,T_i(n))\,(\approx\lo(r\log n))$ computations 
when it is adopted for a continuous support
 model since $|\supp(\fhatt{i}{t})|\le t$. 
On the other hand, $\dmin(F,\mu)$ is differentiable in $\mu$ (with
 slope $\nu^*$) and the
 argument $\mu$ converges to $\mu^*$ after sufficiently many rounds.
Therefore it is reasonable to approximate $\dmin(F,\mu)$ by past
 value of $\dmin(F,\mu;\nu_0,r)$ until the variation of $\mu$ is small. 
In this point of view, we implemented our MED policy for our simulations
 in the following way:
\begin{quote}
{\bf [An implementation of MED policy]}

{\bf Parameter:}\ Integer $r>0$ and real $d>0$. 

{\bf Initialization:}\
\begin{enumerate}
\item Pull each arm  once.
\item Set $(\hat{D}_i,\nu_i):=\dmin(\fhatt{i}{1},\,\muhat^*(K+1);0,r)$ and
 $m_i:=\muhat^*(K+1)$ for each $i=1,\cdots,K$.
\end{enumerate}
{\bf Loop:} For the $n$-th round,
\begin{enumerate}
\item Update variables for each $i$:
\begin{itemize}
\item If $J_{n-1}\neq i$ and $|\muhat^*(n)-m_i|<d$ then
      $\hat{D}_i:=\hat{D}_i+\nu_i(\muhat^*(n)-m_i)$.
\item Otherwise $(\hat{D}_i,\nu_i):=\dmin(\fhatn{i}{n},
\,\muhat^*(n);\nu_i,r)$ and
      $m_i:=\muhatn{i}{n}$.
\end{itemize} 
\item Choose arm $\Pi_j$ according to the
      probability
\begin{eqnarray}
p_j(n)\equiv \frac{\exp(-T_j'(n)\hat{D}_j)}{\sum_{i=1}^K
 \exp(-T_i'(n)\hat{D}_i)}.\n
\end{eqnarray}
\end{enumerate}
\end{quote}

Now we describe the setting of our experiments.
We used MED, UCB-tuned and UCB2.
Each plot is an average over 1,000 different runs.
The parameter $\alpha$ for UCB2 is set to $0.001$, the choice of which
is not very important for the performance (see \cite{ucb}).
First we check the effect of the choice of the parameters $r$ and $d$.
Then MED and UCB policies are compared.

In the following simulations, we use the model where the support is
included in $[0,1]$.
Note that in the computation of $\dmin(F,\mu; \nu,r)$ we assumed that the
support is included in $[-1,0]$ for computational
convenience.
Then, all rewards are passed to computation after $1$ is subtracted
from them in MED.

Table \ref{diss} gives the list of distributions used in the experiments.
They cover various situations on the computation of $\dmin$ and how
distinguishable the optimal arm is.
Distributions 1-4 are examples of 2-armed bandit problems.
In Distribution 1,
$\nu^*\ge \numins$
 in Theorem \ref{thm8} always
holds with equality since $\supp(F_i)\subset\{0,1\}$.
 Therefore the exact solution can be obtained by
$\dmin(F,\mu;\,\nu,r)$ regardless of $r$.
Also in Distribution 2, $\dmin(F,\mu; \nu,r)$ does
not require the repetition after sufficiently many rounds since
$\E_{F_2}[\mu_1/X]<1$.
On the other hand in Distribution 3, the maximization \eqref{opt_ge}
is necessary in almost all rounds since $\E_{F_2}[\mu_1/X]>1$.
Distribution 4 is an example of a difficult problem where
 the optimal arm
is hard to distinguish since the inferior arm appears to be optimal at
first with high probability.
Distribution 5 and 6 are examples of more general problems where
 the numbers of arms $K$ and the support sizes are large.
$\mathrm{Be}(\alpha,\beta)\;\:(\alpha,\beta>0)$
 in Distribution 6 denotes beta
distribution which has the density function
\begin{eqnarray}
\frac{x^{\alpha-1}(1-x)^{\beta-1}}{\mathrm{B}(\alpha,\beta)}\quad
 \mbox{ for }x\in[0,1] \n
\end{eqnarray}
where $\mathrm{B}(\alpha,\beta)$ is beta function.
Note that beta distributions have continuous support and are
 not included in $\fb$ and therefore the
performance of MED is not assured theoretically.
However, MED is still formally applicable since the supports are bounded.

\begin{table}[t]
\caption{Distributions for experiments.}
\label{diss}
\begin{center}
{\normalsize
{\renewcommand\arraystretch{1.2}
\begin{tabular}{|ll|}
\hline
Distribution 1:&\\
$F_1(\{0\})=0.45,\,F_1(\{1\})=0.55$&$\E(F_1)=0.55$\\
$F_2(\{0\})=0.55,\,F_2(\{1\})=0.45$&$\E(F_2)=0.45$\\
\hline
Distribution 2:&\\
$F_1(\{0.4\})=0.5,\,F_1(\{0.8\})=0.5$&$\E(F_1)=0.6$\\
$F_2(\{0.2\})=0.5,\,F_2(\{0.6\})=0.5$&$\E(F_2)=0.4$\\
\hline
Distribution 3:&\\
$F_1(\{x\})=0.08\,$ for
 $x=0,\,0.1,\cdots,0.9$,\;$F_1(\{1\})=0.2$&$\E(F_2)=0.56$\\
$F_2(\{x\})=\frac{1}{11}\,\:\:\;$ for $x=0,\,0.1,\cdots,0.9,\,1$&$\E(F_2)=0.5$\\
\hline
Distribution 4:&\\
$F_1(\{0\})=0.99,\,F_1(\{1\})=0.01$&$\E(F_1)=0.01$\\
$F_2(\{0.008\})=0.5,\,F_2(\{0.009\})=0.5$&$\E(F_2)=0.0085$\\
\hline
Distribution 5:&\\
$F_1(\{x\})=0.08\,$ for
 $x=0,\,0.1,\cdots,0.9$,\;$F_1(\{1\})=0.2$&$\E(F_1)=0.56$\\
$F_i(\{x\})=\frac{1}{11}\,\:\:\;$ for $x=0,\,0.1,\cdots,0.9,\,1$&$\E(F_i)=0.5$\\
&for $i=2,3,4,5$\\
\hline
Distribution 6:&\\
$F_1=\mathrm{Be}(0.9,0.1)$&$\E(F_1)=0.9$\\
$F_2=\mathrm{Be}(7,3)$&$\E(F_2)=0.7$\\
$F_3=\mathrm{Be}(0.5,0.5)$&$\E(F_3)=0.5$\\
$F_4=\mathrm{Be}(3,7)$&$\E(F_4)=0.3$\\
$F_5=\mathrm{Be}(0.1,0.9)$&$\E(F_5)=0.1$\\
\hline
\end{tabular}
}
}
\end{center}
\end{table}

The labels of each figure are as follows.
``regret'' denotes $\sum_{i:\mu_i<\mu^*}(\mu^*-\mu_i)T_i(n)$, which is
the loss due to choosing suboptimal arms.
``\% best arm played'' is the percentage that the best arm is chosen, that
is, $100\times T_1(n)/n$ in these problems.
``Dmin'' stands for the asymptotic bound for a consistent policy,
$\sum_{i:\mu_i<\mu^*}(\mu^*-\mu_i)\log n/\dmin(F_i,\mu^*)$.
The asymptotic slope of the regret (in the semi-logarithmic plot)
 of a consistent policy is more than or equal to that of ``Dmin''.

\begin{figure}[tb]
 \begin{minipage}{0.5\hsize}
  \begin{center}
   \includegraphics[bb=86 210 503 590,clip,width=65mm]{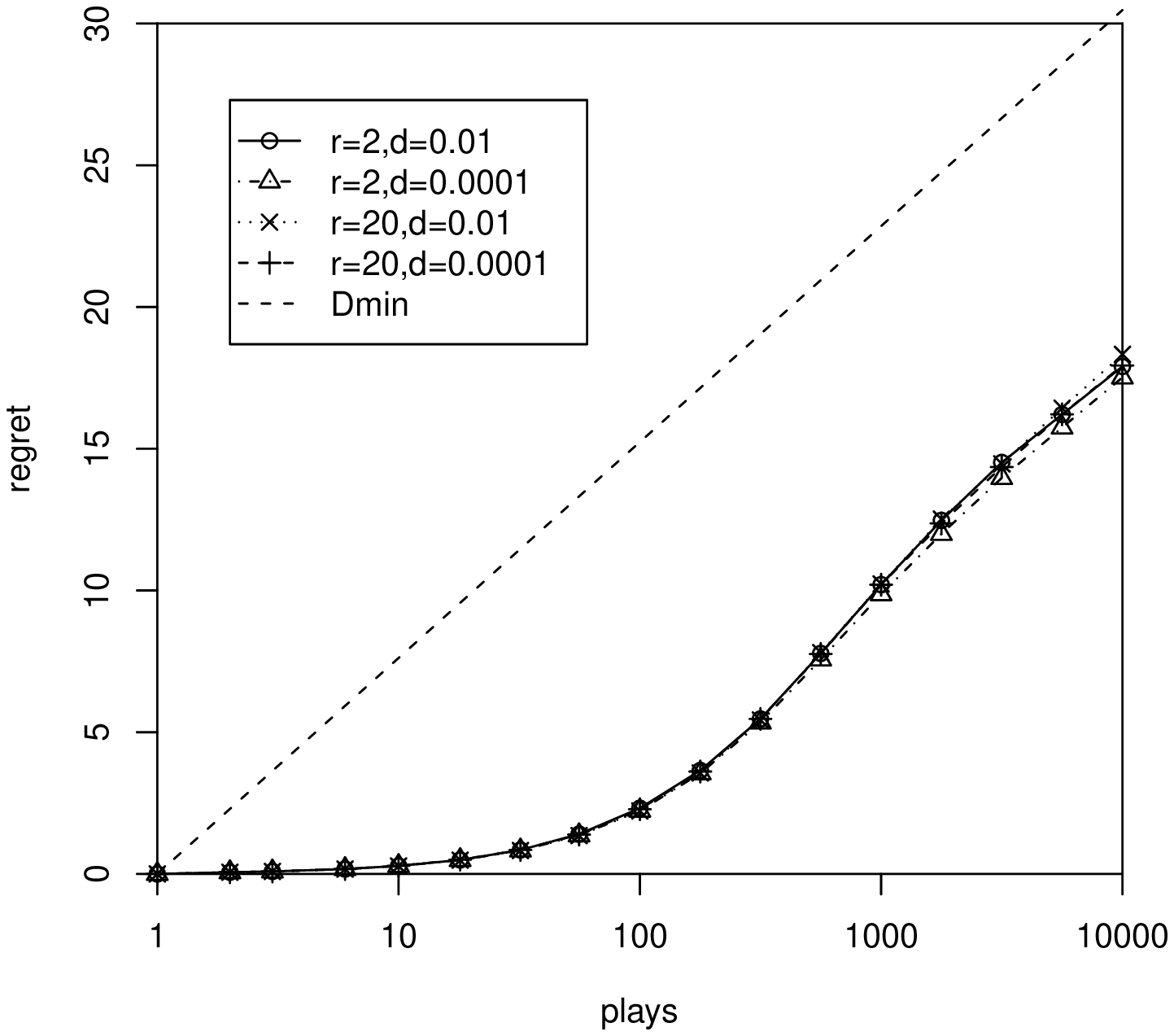}
  \end{center}
 \end{minipage}
 \begin{minipage}{0.5\hsize}
  \begin{center}
   \includegraphics[bb=86 210 503 590,clip,width=65mm]{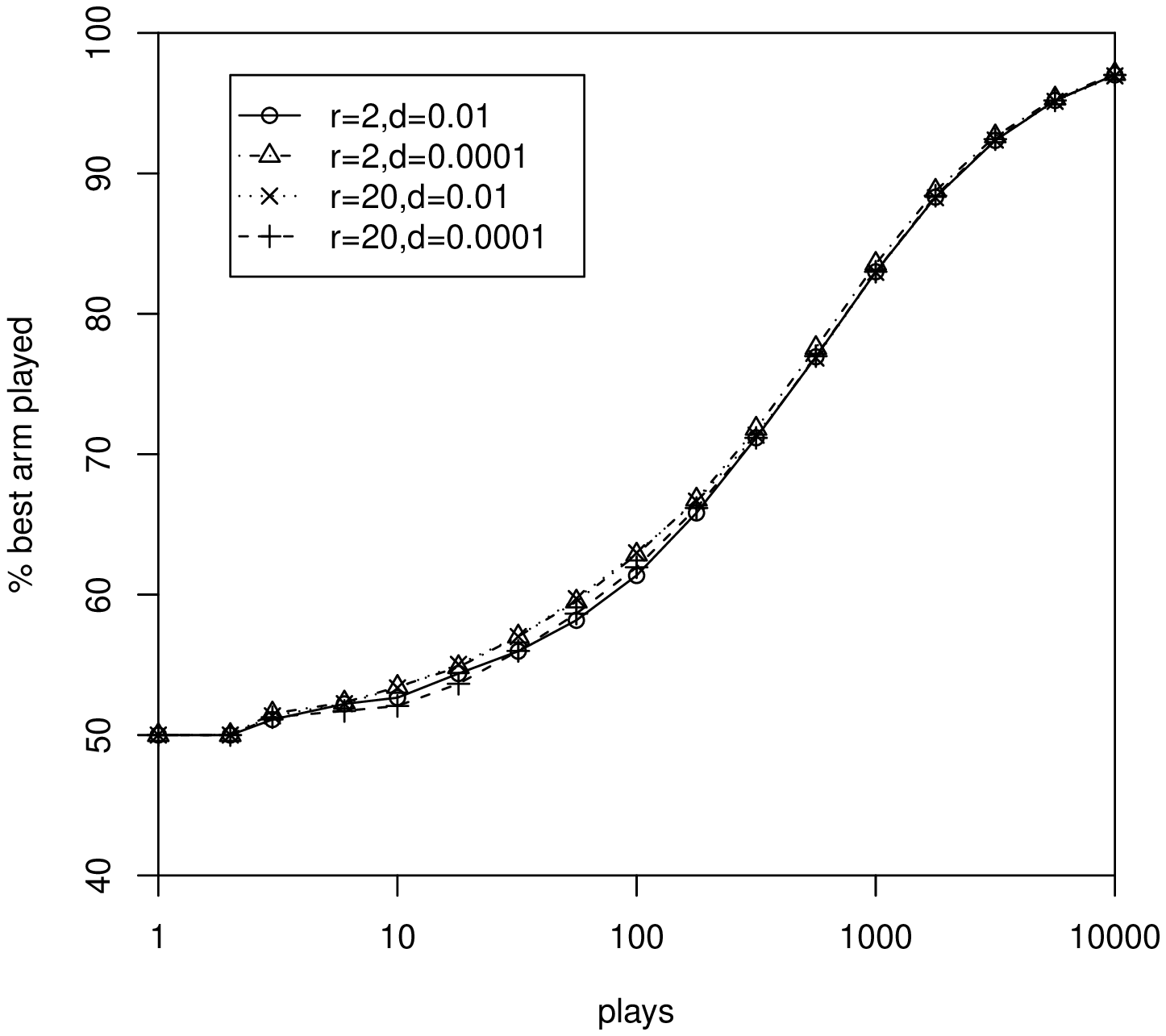}
  \end{center}
 \end{minipage}
 \caption{Comparison between different parameters of MED.}
 \label{ex1}
\end{figure}

Figure \ref{ex1} shows an experiment on the choice of the parameters $r$
and $d$ of MED for Distribution 3.
Our implementation of MED approaches the ideal MED as $d\to 0$ and
$r\to\infty$.
However, we see from the figure that the performance is not sensitive to
the choice of $r,\,d$.
This may be understood as follows: (1) the linear approximation for
the case $|\muhat^*(n)-m_i|<d$ is accurate, (2) the initial value $\nu_i$
in $\dmin(\fhatn{i}{n},\,\muhat^*(n);\nu_i,r)$ seems to be a good
approximation of $\nu^*$ and the repetition
number does not have to be large. 
We use $r=2$ and $d=0.01$ in the remaining experiments based on this result.

Now we summarize the remaining experiments on the comparison of the
policies (Figure 2--7).
\begin{itemize}
\item MED always seems to be achieving the asymptotic bound even
      for continuous support distributions, since 
the asymptotic slope of the regret is close to that of ``Dmin''.
\item MED performs best except for Distribution 1 where MED performs
      worst.
However, the consistency of UCB-tuned is not proved unlike
      MED and UCB2.
It appears that UCB-tuned might not be consistent, because
the asymptotic slope of $T_2(n)$ seems to be smaller than that of ``Dmin''.
Note that the theoretical logarithmic term of the regret is very near
      between MED and UCB2 for Distribution 1 ($4.983\log n$ and
      $5.025\log n$, respectively).
Therefore this result can be interpreted as follows: MED achieves the
      asymptotic bound but needs some improvement in the constant term
      of the regret compared to UCB2.
\end{itemize}
\begin{figure}[tbhp]
 \begin{minipage}{0.5\hsize}
  \begin{center}
   \includegraphics[bb=86 220 503 590,clip,width=65mm]{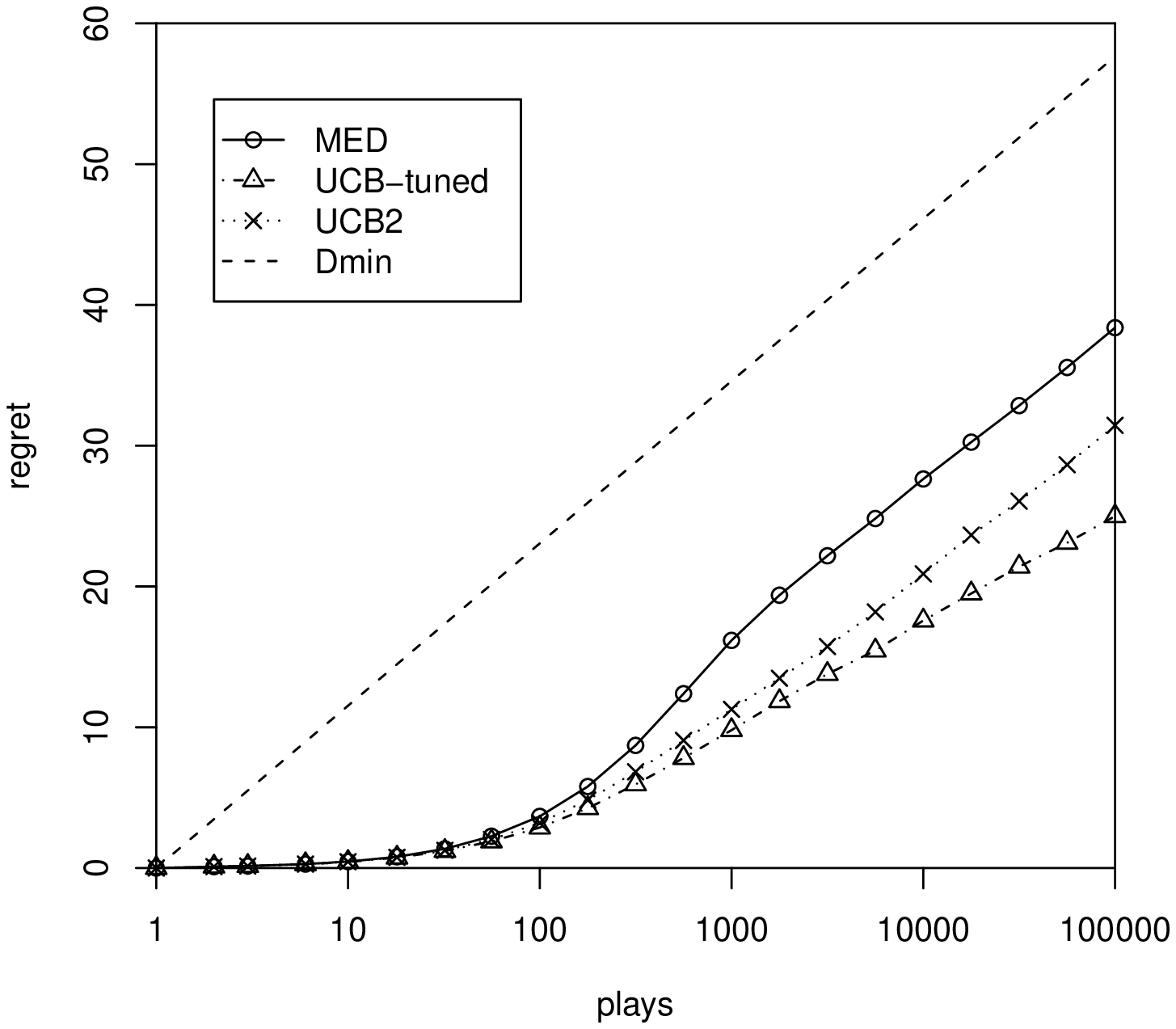}
  \end{center}
 \end{minipage}
 \begin{minipage}{0.5\hsize}
  \begin{center}
   \includegraphics[bb=86 220 503 590,clip,width=65mm]{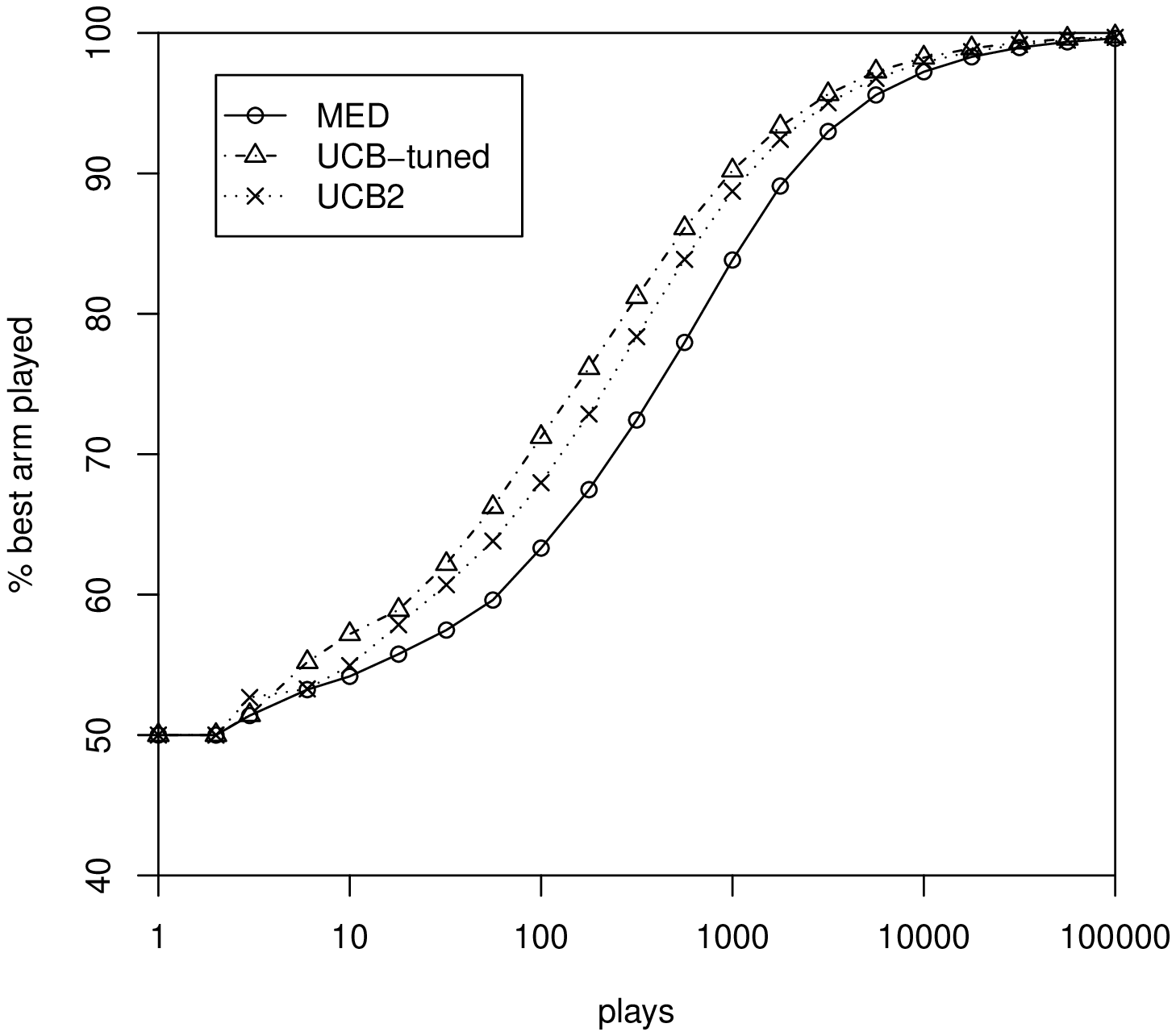}
  \end{center}
 \end{minipage}\vspace{-2mm}
 \caption{Simulation result for Distribution 1 (Bernoulli distributions).
}
 \label{ex2}
\vspace{7mm}
 \begin{minipage}{0.5\hsize}
  \begin{center}
   \includegraphics[bb=86 220 503 590,clip,width=65mm]{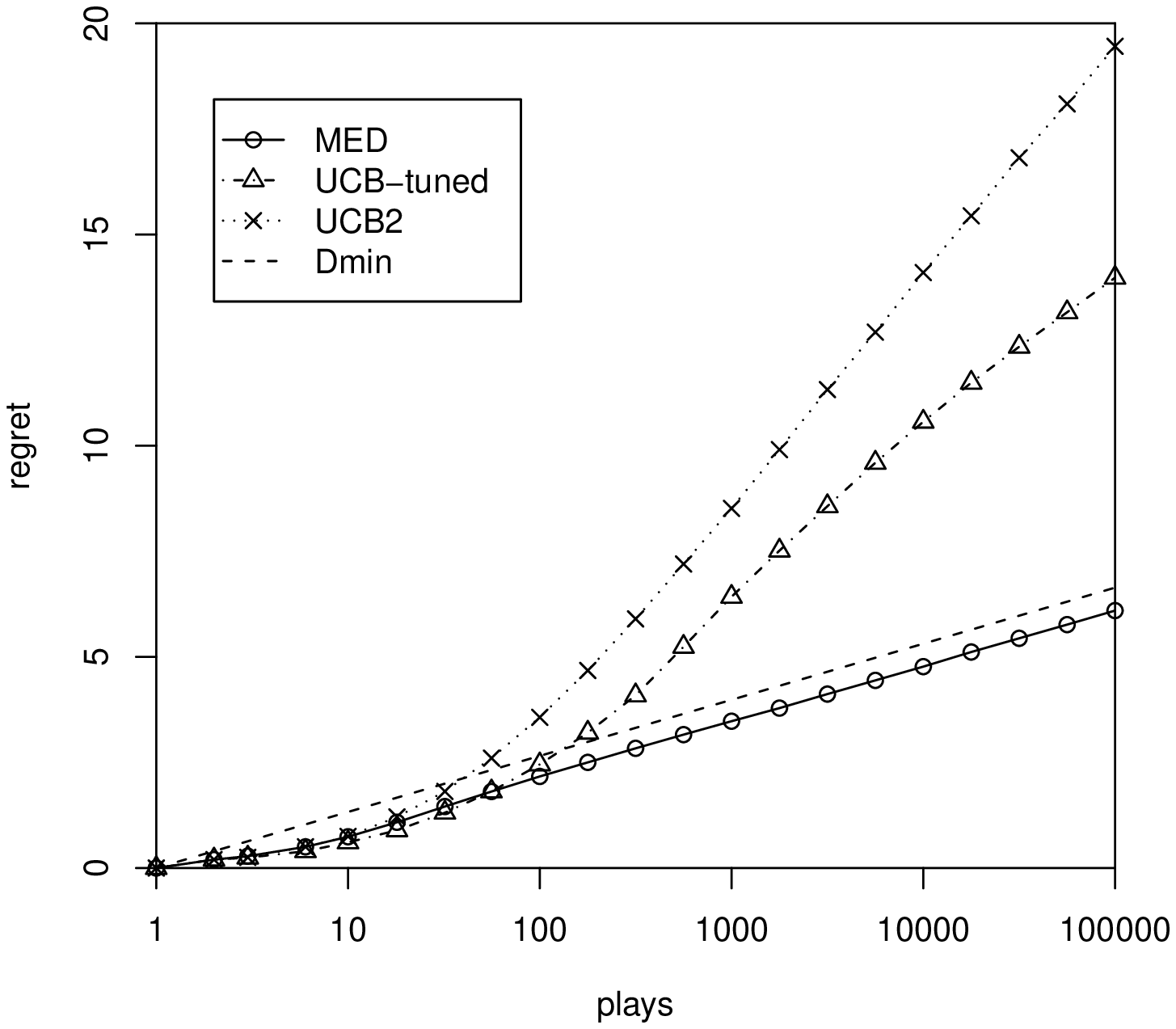}
  \end{center}
 \end{minipage}
 \begin{minipage}{0.5\hsize}
  \begin{center}
   \includegraphics[bb=86 220 503 590,clip,width=65mm]{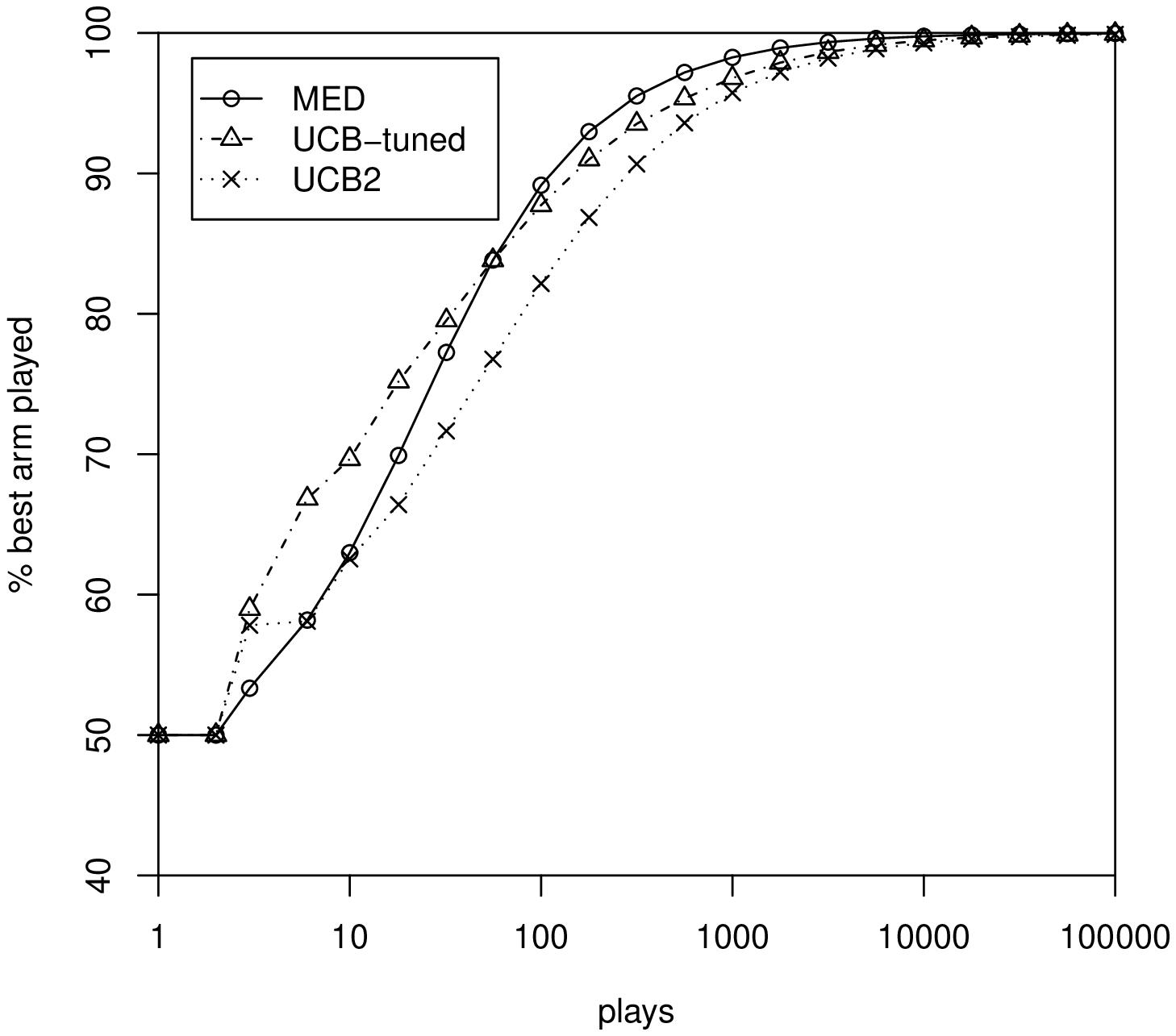}
  \end{center}
 \end{minipage}\vspace{-2mm}
\caption{Simulation result for Distribution 2 (uniform distributions with
 different supports).}
 \label{ex3}
\vspace{7mm}
 \begin{minipage}{0.5\hsize}
  \begin{center}
   \includegraphics[bb=86 220 503 595,clip,width=65mm]{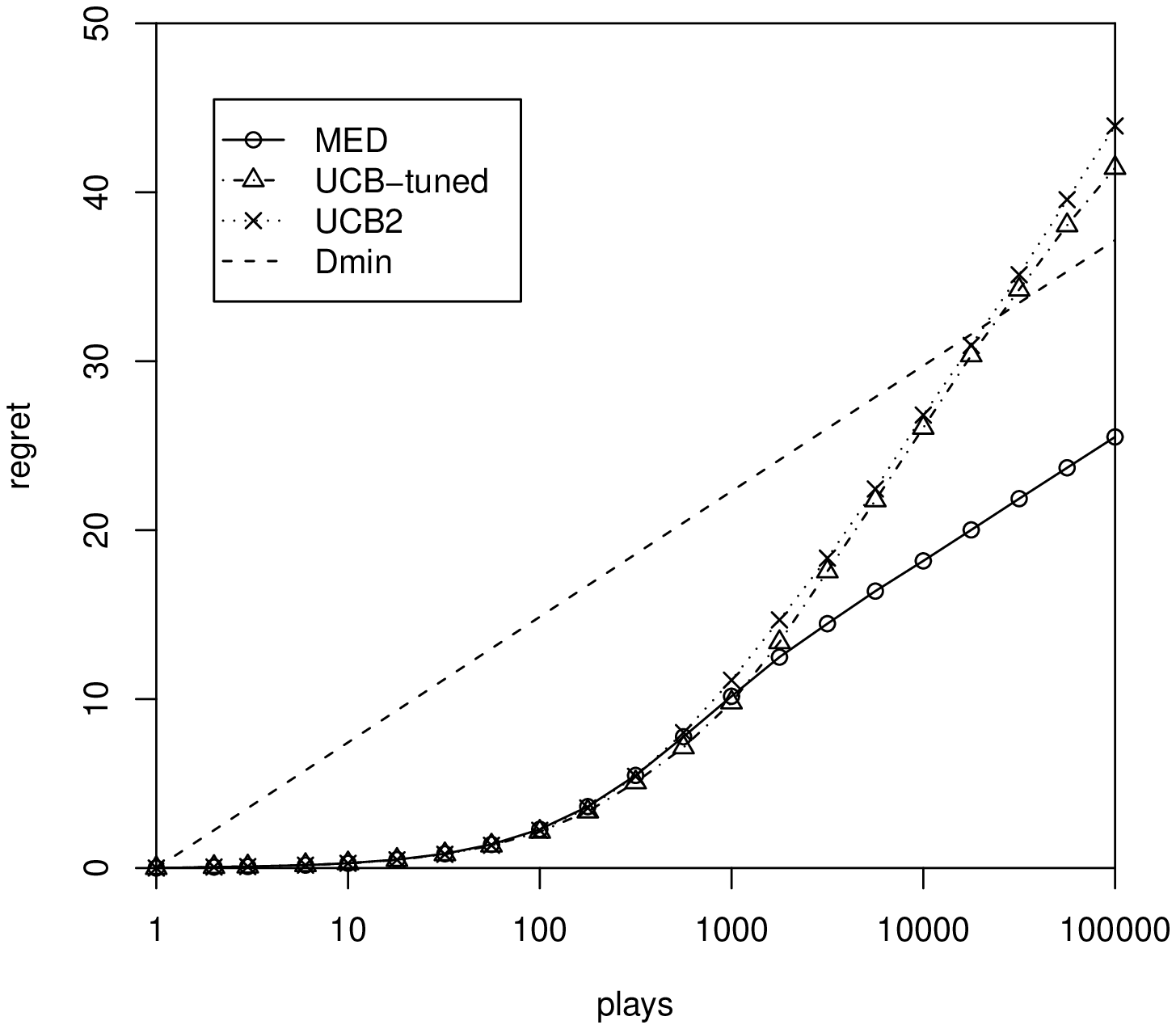}
  \end{center}
 \end{minipage}
 \begin{minipage}{0.5\hsize}
  \begin{center}
   \includegraphics[bb=86 220 503 595,clip,width=65mm]{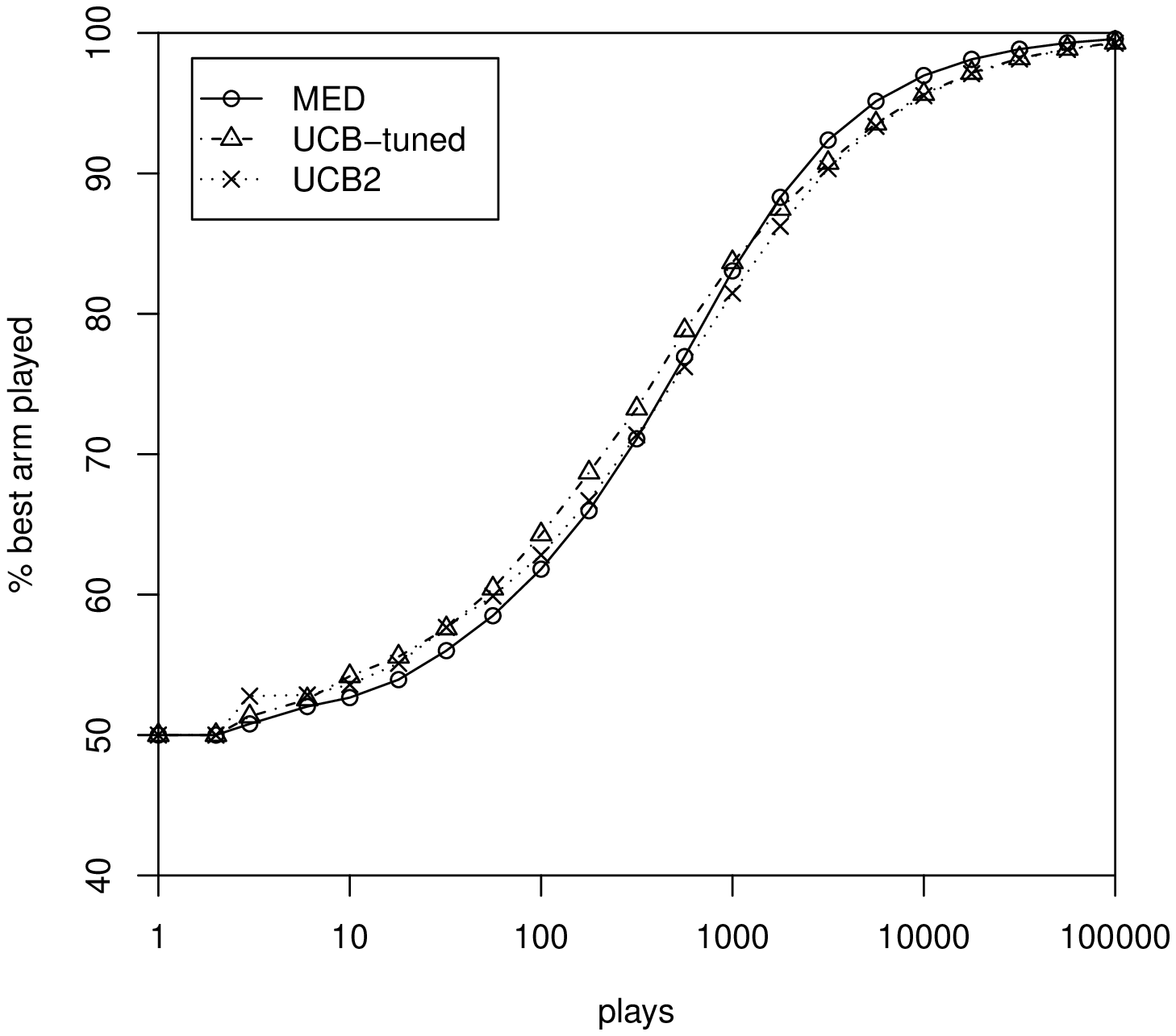}
  \end{center}
 \end{minipage}\vspace{-2mm}
 \caption{Simulation result for Distribution 3 (distributions 
where $\dmin$ is computed by repetitions
).}
 \label{ex4}\vspace{-7mm}
\end{figure}

\begin{figure}[tbhp]
 \begin{minipage}{0.5\hsize}
  \begin{center}
   \includegraphics[bb=86 220 503 595,clip,width=65mm]{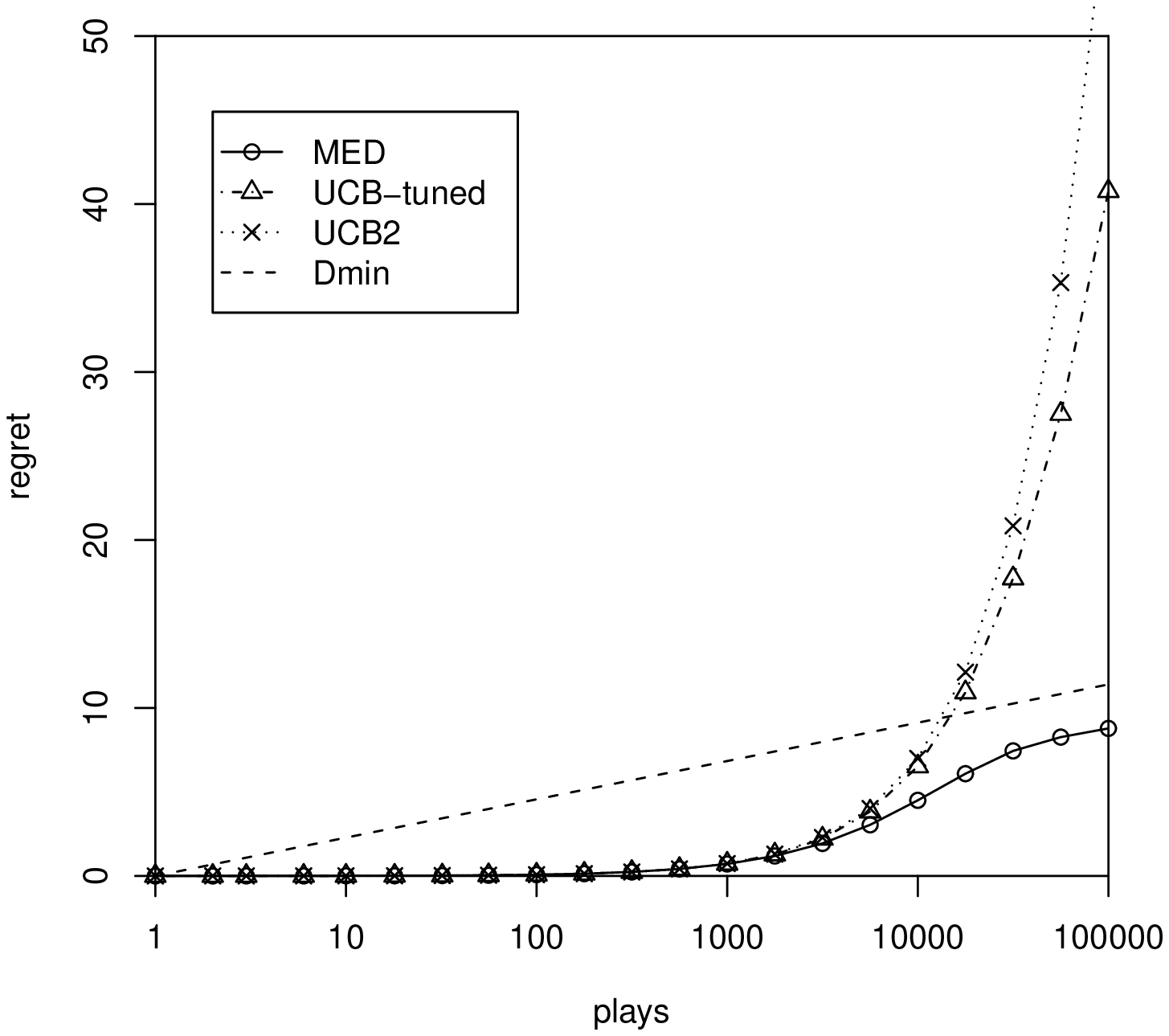}
  \end{center}
 \end{minipage}
 \begin{minipage}{0.5\hsize}
  \begin{center}
   \includegraphics[bb=86 220 503 595,clip,width=65mm]{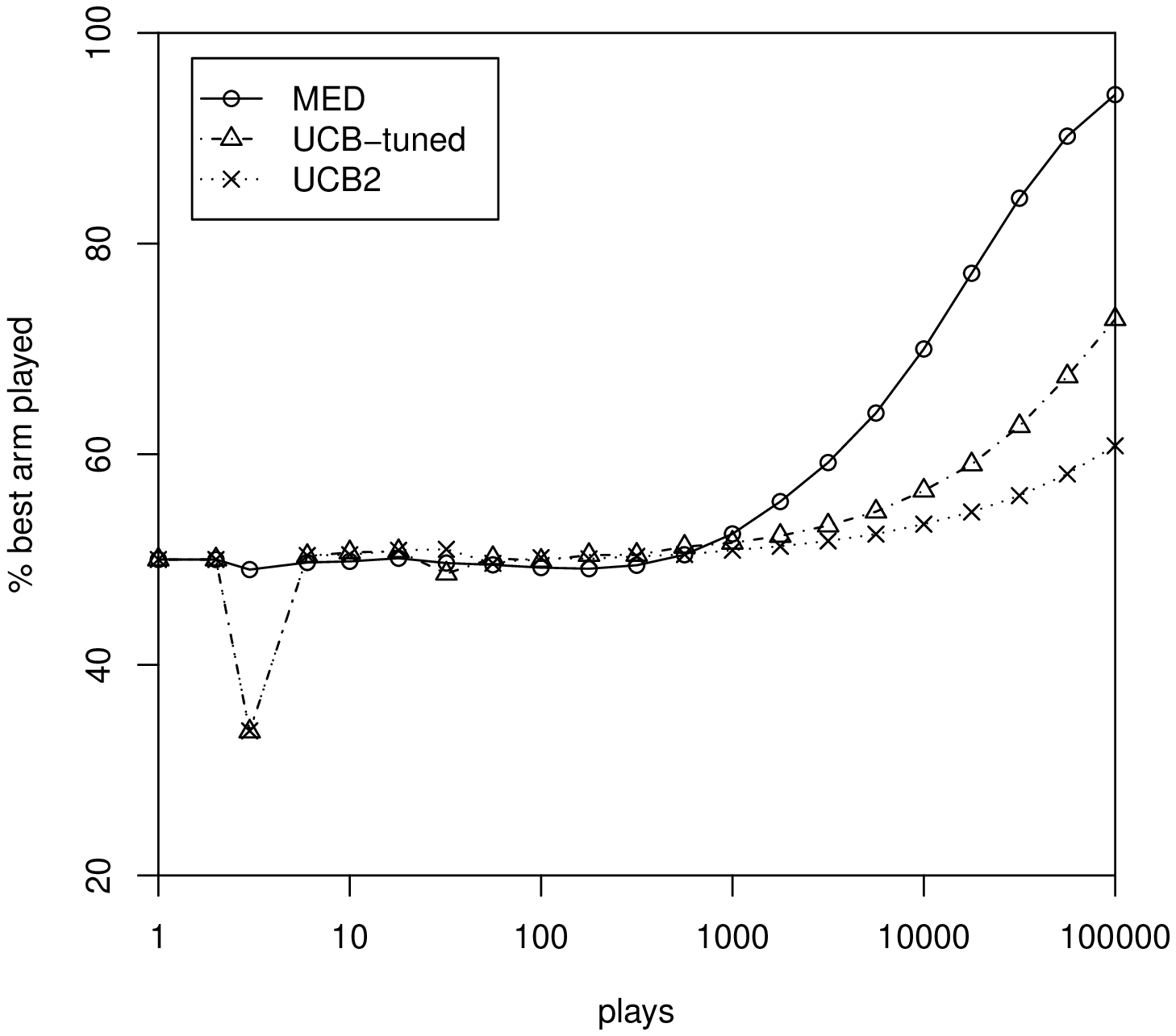}
  \end{center}
 \end{minipage}\vspace{-2mm}
 \caption{Simulation result for Distribution 4 (very confusing distributions).}
 \label{ex5}
\vspace{7mm}
 \begin{minipage}{0.5\hsize}
  \begin{center}
   \includegraphics[bb=86 220 503 595,clip,width=65mm]{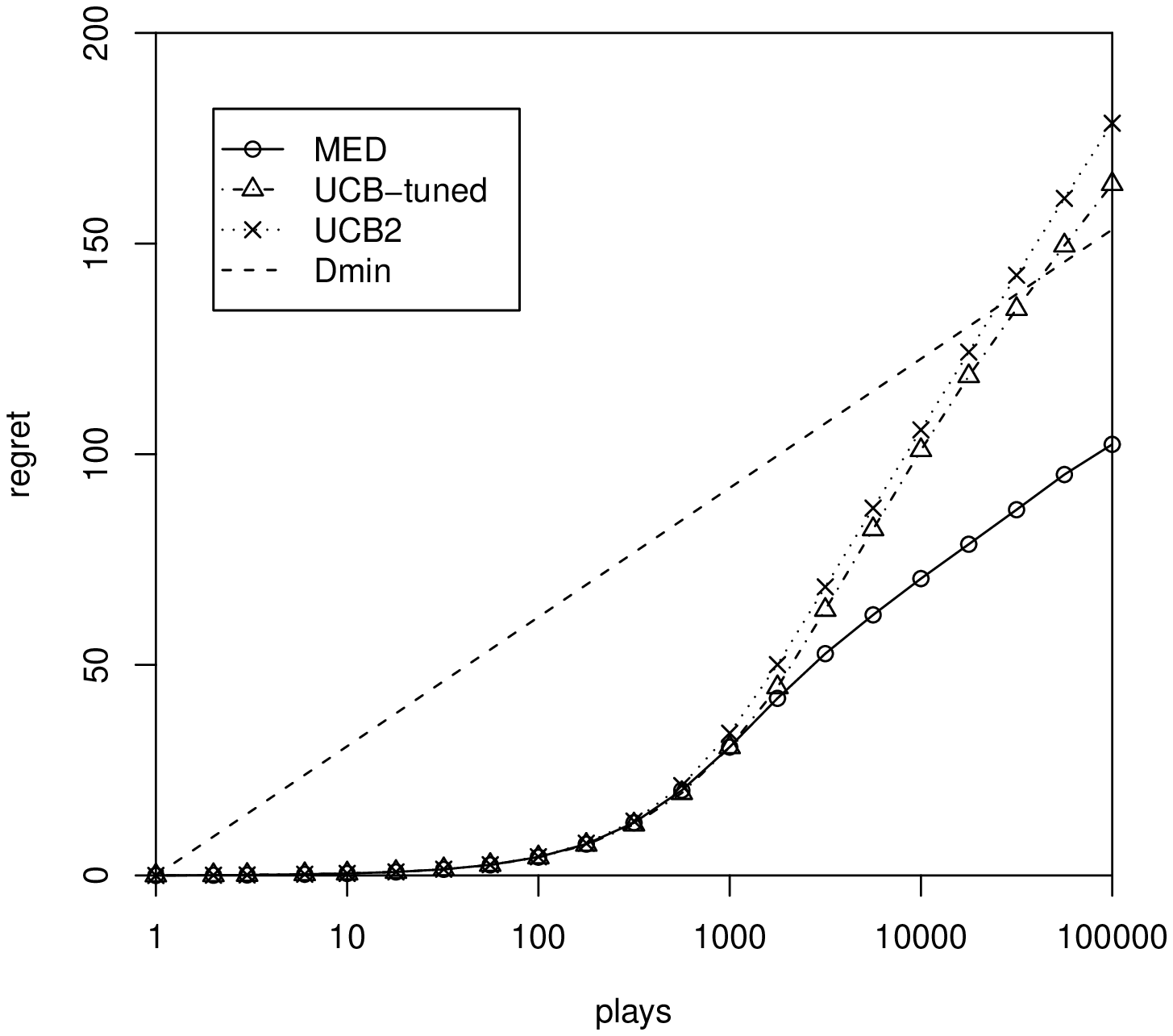}
  \end{center}
 \end{minipage}
 \begin{minipage}{0.5\hsize}
  \begin{center}
   \includegraphics[bb=86 220 503 595,clip,width=65mm]{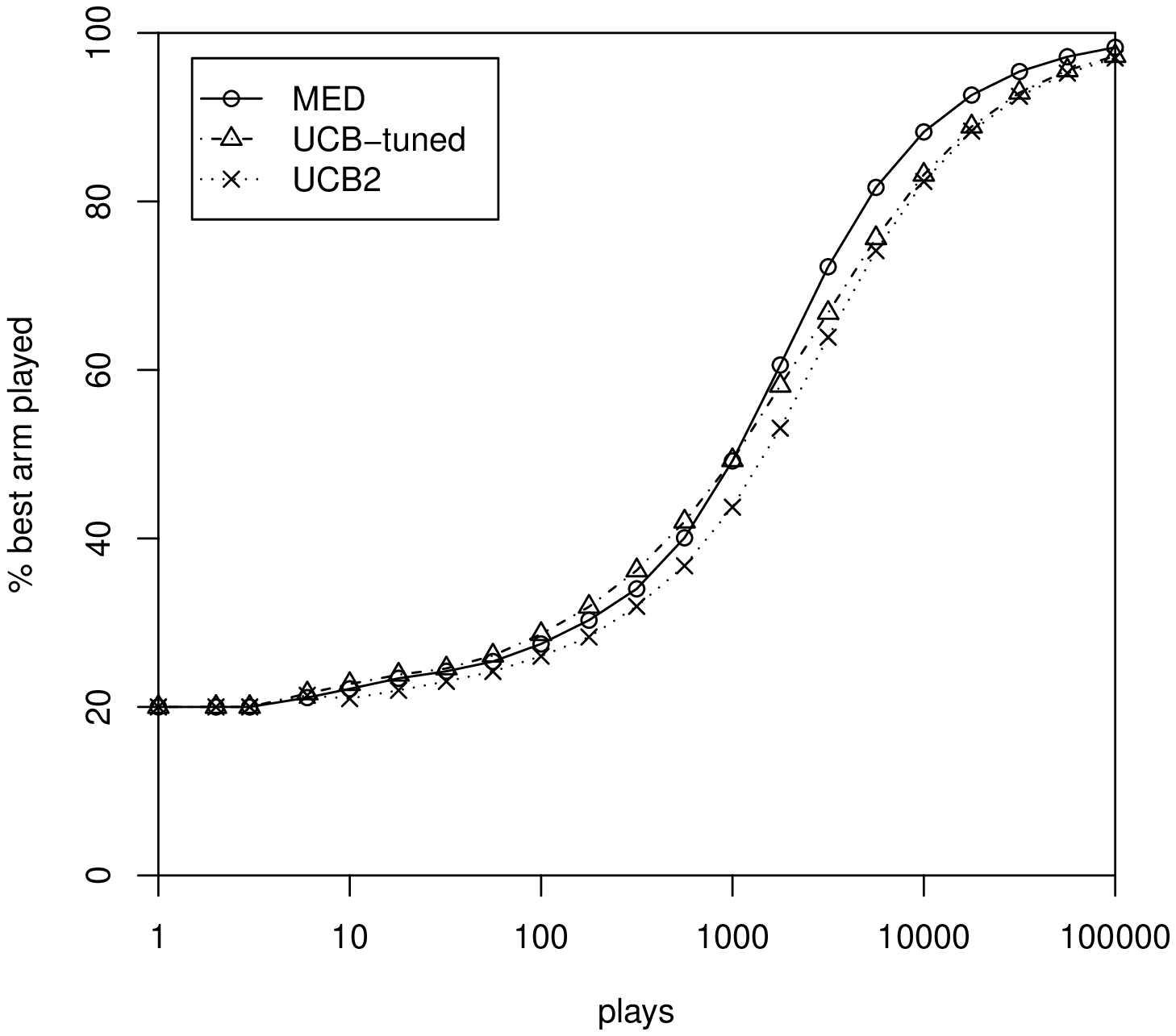}
  \end{center}
 \end{minipage}\vspace{-2mm}
 \caption{Simulation result for Distribution 5 (5 arms with a wide support).}
 \label{ex6}
\vspace{7mm}
 \begin{minipage}{0.5\hsize}
  \begin{center}
   \includegraphics[bb=86 220 503 595,clip,width=65mm]{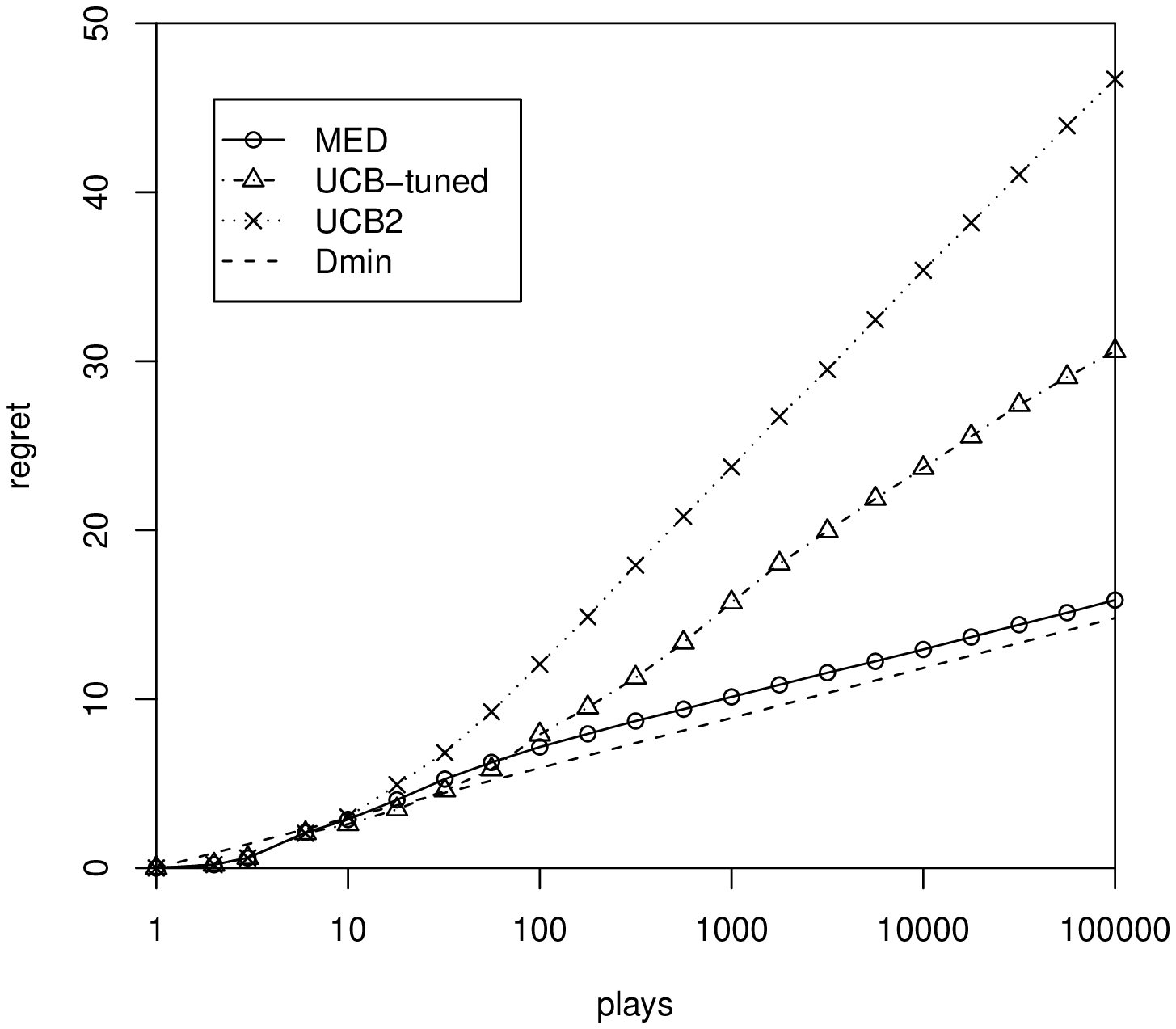}
  \end{center}
 \end{minipage}
 \begin{minipage}{0.5\hsize}
  \begin{center}
   \includegraphics[bb=86 220 503 595,clip,width=65mm]{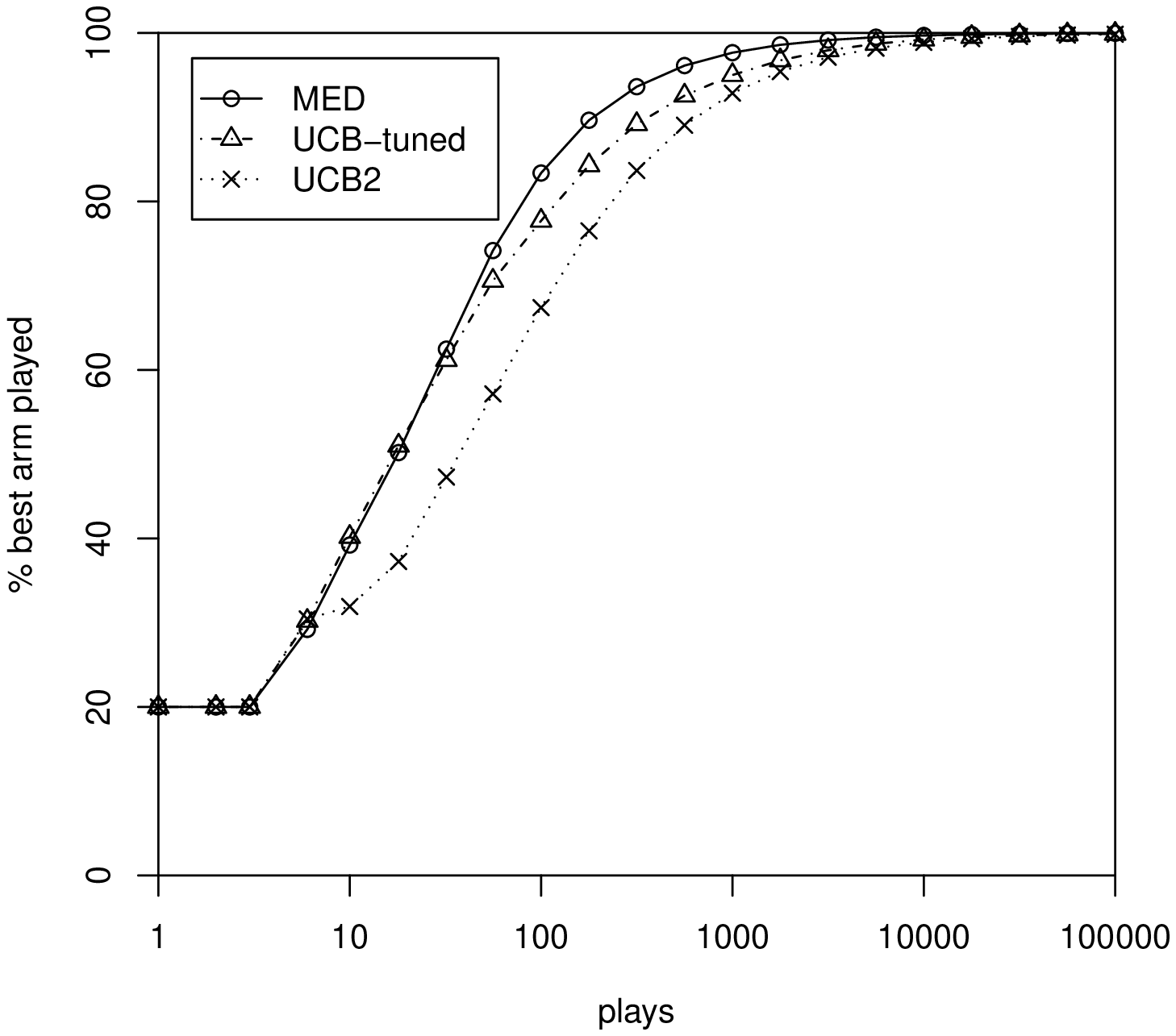}
  \end{center}
 \end{minipage}\vspace{-2mm}
 \caption{Simulation result for Distribution 6 (beta distributions).}
 \label{ex7}
\end{figure}

\section{Concluding remarks}
\label{section-remarks}
We proposed a policy, MED, 
and proved that our policy achieves the asymptotic bound 
for finite support models.  We also showed that
our policy can be implemented efficiently by a convex optimization
technique.

In the theoretical analysis of this paper, we assumed the finiteness of
the support although MED worked nicely
also for distributions with continuous bounded support in the simulation.
We conjecture that the optimality of MED holds also for the continuous bounded
support model.
In addition, there are many models that $\dmin$ can be
 computed explicitly,
 such as normal distribution model with unknown mean and
variance.
We expect that our MED can be extended to these models. 
Furthermore, our MED is a randomized policy and the theoretical evaluation
of the expectation includes randomization in the policy.
We may be able to construct a deterministic version of MED.

In addition to the above theoretical analyses, it is also important
to consider the finite horizon case.
Then it is necessary to derive a finite-time bound of MED for this case.
Especially, MED policy itself should be improved when the
number of rounds is given in advance.
In this setting, the value of ``exploration'' becomes
smaller and a current best arm is to be pulled more often as the
number of remaining rounds becomes smaller.

\bibliographystyle{plain}

\end{document}